\documentclass{siamart220329}
\makeatletter
\def\refstepcounter@noarg#1{%
  \cref@old@refstepcounter{#1}%
  \cref@constructprefix{#1}{\cref@result}%
  \@ifundefined{cref@#1@alias}%
    {\def\@tempa{#1}}%
    {\def\@tempa{\csname cref@#1@alias\endcsname}}%
  \protected@edef\cref@currentlabel{%
    [\@tempa][\arabic{#1}][\cref@result]%
    \csname p@#1\endcsname\csname the#1\endcsname}}%
\def\refstepcounter@optarg[#1]#2{%
  \cref@old@refstepcounter{#2}%
  \cref@constructprefix{#2}{\cref@result}%
  \@ifundefined{cref@#1@alias}%
    {\def\@tempa{#1}}%
    {\def\@tempa{\csname cref@#1@alias\endcsname}}%
  \protected@edef\cref@currentlabel{%
    [\@tempa][\arabic{#2}][\cref@result]%
    \csname p@#2\endcsname\csname the#2\endcsname}}%
\makeatother
\usepackage{amsmath,amssymb}
\usepackage{mathtools}
\usepackage{epsfig,graphics}
\usepackage{accents}

\usepackage{color}
\usepackage{xcolor}
\usepackage{comment}
\usepackage{url}

\usepackage{tabularx}
\usepackage{diagbox}

\usepackage{array,booktabs}
\usepackage{siunitx}
\usepackage{float}
\usepackage{algorithm}
\usepackage{algpseudocode}

\usepackage{enumitem}

\newcommand{\Bk}{}
\newcommand{\Rd}{}
\newcommand{\Bl}{}

\newtheorem{remark}[theorem]{Remark}

\newcommand{\be}{\begin{equation}}
\newcommand{\ee}{\end{equation}}
\newcommand{\ba}{\begin{aligned}}
\newcommand{\ea}{\end{aligned}}

\DeclareMathOperator{\Span}{span}

\def\Re{\mbox{Re}}
\def\Im{\mbox{Im}}

\newcommand{\br}{{\boldsymbol r}} 
\newcommand{\brp}{{\boldsymbol r}'} 

\newcommand{\bD}{{\boldsymbol D}}
\newcommand{\bS}{{\boldsymbol S}}

\newcommand{\slp}{{{\mathcal{S}}}} 
\newcommand{\dlp}{{{\mathcal{D}}}}
\newcommand{\sprime}{{{\mathcal{S}'}}}
\newcommand{\dprime}{{{\mathcal{D}'}}}

\newcommand{\hlm}{H^{(l,m)}} 
\newcommand{\htlm}{\tilde{H}^{(l,m)}}
\newcommand{\nhlm}{\nabla H^{(l,m)}}

\newcommand{\slm}{S^{(l,m)}} 
\newcommand{\stlm}{\accentset{\circ}{S}^{(l,m)}} 
\newcommand{\nslm}{\nabla S^{(l,m)}}
\newcommand{\ntslm}{\accentset{\circ}{\nabla} S^{(l,m)}} 

\newcommand{\rlm}{R_l^{m}} 
\newcommand{\ylm}{Y_l^{m}} 
\newcommand{\plm}{P_l^{m}} 
\newcommand{\Flm}{F^{(l,m)}}

\newcommand{\bi}{{\boldsymbol i}}
\newcommand{\bj}{{\boldsymbol j}}
\newcommand{\bk}{{\boldsymbol k}}
\newcommand{\bu}{{\boldsymbol u}}
\newcommand{\bv}{{\boldsymbol v}}
\newcommand{\bA}{{\boldsymbol A}}
\newcommand{\bB}{{\boldsymbol B}}

\newcommand{\bzero}{{\boldsymbol 0}} 
\newcommand{\np}{n_{\rm p}} 

\newcommand{\bp}{{\boldsymbol p}}
\newcommand{\tp}{{\tilde{p}}}
\newcommand{\tc}{{\tilde{c}}}

\newcommand{\dlptform}{\alpha^{(l,m)}} 
\newcommand{\dlpoform}{\beta^{(l,m)}} 
\newcommand{\cdlptform}{\tilde{\alpha}^{(l,m)}} 
\newcommand{\cdlpoform}{\tilde{\beta}^{(l,m)}} 
\newcommand{\gammalm}{\gamma^{(l,m)}} 
\newcommand{\bomega}{\boldsymbol \omega} 
\newcommand{\bOmega}{\boldsymbol \Omega} 

\newcommand{\slptform}{\chi^{(l,m)}} 
\newcommand{\slpoform}{\psi^{(l,m)}} 
\newcommand{\cslptform}{\tilde{\chi}^{(l,m)}} 
\newcommand{\cslpoform}{\tilde{\psi}^{(l,m)}} 

\newcommand{\clm}{c^{(l,m)}} 
\newcommand{\dlm}{d^{(l,m)}} 
\newcommand{\glm}{g^{(l,m)}} 
\newcommand{\wlm}{w^{(l,m)}} 
\newcommand{\Dlm}{D^{(l,m)}} 

\newcommand{\bclm}{{\bf c}^{(l,m)}}

\newcommand{\f}{{\boldsymbol f}} 
\newcommand{\g}{{\boldsymbol g}} 
\newcommand{\h}{{\boldsymbol h}} 

\newcommand{\qjk}{{q^{(j,k)}}}
\newcommand{\Qjk}{{Q^{(j,k)}}}

\newcommand{\vjk}{{v^{(j,k)}}}
\newcommand{\Vjk}{{V^{(j,k)}}}

\newcommand{\bnu}{\boldsymbol{\nu}} 

\def\ccm{Center for Computational Mathematics, Flatiron Institute, Simons Foundation,
  New York, New York 10010}

\def\papertitle{Recursive reduction quadrature for the evaluation of Laplace layer
  potentials in three dimensions}

\title{\papertitle}

\author{Shidong Jiang\thanks{\ccm\,({\tt sjiang@flatironinstitute.org})}
  \and Hai Zhu\thanks{Institute for Theoretical Sciences, Westlake University, Hangzhou, China\,({\tt hzhu@westlake.edu.cn})}}

\begin{document}

\maketitle

\begin{abstract}
  A high-order quadrature scheme is constructed for the evaluation of Laplace single
  and double layer potentials and their normal derivatives on smooth surfaces in three
  dimensions. The construction begins with a harmonic approximation of the density
  {\it on each patch}, which allows for a natural harmonic polynomial extension in a
  {\it volumetric neighborhood} of the patch in the ambient space. Then by the
  general Stokes theorem,
  singular and nearly singular surface integrals are reduced to line integrals
  preserving the singularity of the kernel, instead of the standard origin-centered
  1-forms that require expensive adaptive integration. These
  singularity-preserving line integrals can be semi-analytically evaluated using
  singularity-swap quadrature. In other words, the evaluation of singular and nearly
  singular surface integrals is reduced to function evaluations at the vertices on
  the boundary of each patch. The recursive reduction quadrature largely removes
  adaptive integration that is needed in most existing high-order quadratures for
  singular and nearly singular surface integrals, resulting in exceptional performance.
  The scheme achieves twelve-digit accuracy uniformly for close evaluations and
  offers a speedup of five times or more in constructing the sparse
  quadrature-correction matrix compared to previous state-of-the-art quadrature schemes.
\end{abstract} 



\pagestyle{myheadings}
\thispagestyle{plain}
\markboth{S. Jiang and H. Zhu}
{Recursive reduction quadrature for the Laplace layer potentials in 3D}

\section{Introduction}\label{sec:intro}
We consider the evaluation of Laplace layer potentials on smooth surfaces
in three dimensions. The Laplace single and double layer potentials and their
normal derivatives are defined by the formulas
\begin{align}
\slp[\sigma](\brp)&=\int_S G(\brp,\br)\sigma(\br)da(\br),\label{slpdef}\\
\dlp[\mu](\brp)&=\int_S \frac{\partial G(\brp,\br)}{\partial \bnu} \mu(\br)da(\br),\label{dlpdef}\\
\sprime[\sigma](\brp)&=\int_S \frac{\partial G(\brp,\br)}{\partial \bnu'}
\sigma(\br)da(\br),
\label{sprimelpdef}\\
\dprime[\mu](\brp)&=\int_S \frac{\partial^2 G(\brp,\br)}{\partial \bnu'\partial \bnu} \mu(\br)da(\br),\label{dprimelpdef}
\end{align}
where $G$ is the Laplace Green's function 
\be\label{greenfunction}
G(\brp,\br)=\frac{1}{4\pi |\brp-\br|}, \quad \brp, \br \in \mathbb{R}^3.
\ee
Here, $da(\br)$ is the surface differential at $\br$. The point
$\brp$ is referred to as the target point, while $\br$ is the source point.
The symbols $\bnu'$ and $\bnu$ represent the unit outward normal vectors at $\brp$
and $\br$, respectively.
We assume that the densities $\sigma$ and $\mu$ are sufficiently smooth,
and that the surface $S$ is piecewise
smooth and divided into a collection of patches,
where each patch is sufficiently smooth and admits a high-order parameterization.
Thus, the only singularity in the above integrals
is induced by the Green's function or its derivatives.

When integral equation methods are applied to solve boundary value problems
(BVPs) of elliptic partial differential equations, the solution is represented
using layer potentials, and sometimes volume potentials as well. Imposing the
boundary conditions leads to boundary integral equations (BIEs). These BIEs
are discretized, and the resulting linear system is solved using either
iterative solvers such as GMRES~\cite{gmres}, accelerated by fast algorithms
like the fast multipole method (FMM)~\cite{greengard1987thesis,greengard1988,greengard1987jcp,carrier1988sisc,fmm2,fmm3,fmm4,fmm6,fmm7,fmm8},
or fast direct solvers~\cite{fds1,rskelf1,fds16,fds2,fds3,kenho1,ho2016cpam1,ho2016cpam2,fds5,fds6,fds7,fds8,minden2017mms,rskelf2}.
Finally, layer potentials are evaluated to obtain the solution throughout the domain.
Thus, layer potentials are fundamental building blocks in fast integral
equation methods: they are used in both the solve phase and the evaluation phase.
Their accuracy greatly affects the accuracy of the BVP solution, and their efficiency
is often a bottleneck in the solver's overall speed.

Many engineering applications involve complex geometries where mesh generation -- such
as triangulation -- is the initial step in discretization. In such cases, a
\emph{local quadrature} method divides the entire surface into a collection of
elements with finite geometries (also referred to as chunks, panels, or patches
in the literature) and then constructs quadrature nodes and weights for each element.
Consequently, one only needs to consider the evaluation of Laplace layer potentials
\cref{slpdef} -- \cref{dprimelpdef} when the surface $S$ is replaced by one of these
patches, denoted as $P$. We will primarily focus on triangular patches to avoid
diverting the reader from the main ideas with technical details specific to
non-triangular patches, even though our quadrature scheme essentially works for
patches of \emph{arbitrary shape}.

Depending on the distance of the target to the source patch, there are altogether
four cases:
\begin{enumerate}[label=(\alph*)]
\item Smooth interaction: the target $\brp$ is sufficiently far from the source element.
\item Self interaction: the target $\brp$ is on the source element.
\item Near interaction: the target $\brp$ is not on the source element, but on $S$ and
  close to the source element.
\item Close evaluation: the target $\brp$ is off-surface and close to the source element.
\end{enumerate}
In the case of smooth interaction, the integrals are smooth, and high-order quadrature can
be used to evaluate layer potentials in both two and three dimensions.
For case (b), the integrals are weakly singular (i.e., integrable in the Riemann sense,
mostly with logarithmic singularity in two dimensions or $1/$ singularity in
three dimensions for layer potentials), singular (i.e., the so-called principal value
integrals), or hypersingular (i.e., the so-called Hadamard finite-part integrals).
With a slight abuse of terminology, we will simply refer to case (b) as singular
integrals. For cases (c) and (d), the integrals are nearly weakly singular, nearly
singular, or nearly hypersingular. With a slight abuse of terminology, we will refer
to these two cases as near-singular integrals. The difference between case (c) and
case (d) is that the target point is associated with a parameter value in the parameter
space for $S$ in case (c), while the target point only has coordinates in the
underlying physical space in case (d).
Somewhat counterintuitive to those unfamiliar with integral equation methods,
case (b) is actually not so hard to handle. Since the target is on the source element,
various techniques can be employed to remove or cancel out the singularity in the
kernels. Case (d) proves to be the most challenging
because the target lies in the ambient space and lacks an associated parameter value,
rendering many singularity subtraction or removal techniques inapplicable.

Over the past fifty years or so, numerous researchers have tackled this important
problem. In \cite{duan2009high,gimbutas2013fast,graham2002fully,wu2021nm,wu2021acom,wu2023sinum}, high-order global quadratures have been developed for
surfaces homeomorphic to the sphere or torus. In \cite{duffy} and subsequently
\cite{bruno1,bruno2}, polar coordinates centered at the given target point
are used to eliminate the $1/r$ singularity for self interactions. In
\cite{gimbutas2012jcp,gimbutas2013jcp},
{\it generalized Gaussian quadrature} (GGQ)~\cite{ggq1,ggq2,ggq3}
is applied along each radial direction in polar coordinates to achieve
better efficiency and accuracy, capable of handling triangles with aspect ratio
as high as a million.
For cases (c) and (d), a popular approach is to evaluate layer
potentials for targets sufficiently away from the surface using oversampled smooth
quadrature rule, then apply interpolation/expolation to evaluate the layer potential
at the given target point. This includes the QBX {\it (quadrature by expansions)}~\cite{qbx2,qbx1,rachh2017jcp,siegel2018jcp,wala2019jcp,wala2020jcp},
its descendants~\cite{ding2021quadrature,morse2021robust,stein2022quadrature},
and methods in \cite{beale2004grid,beale2024extrapolated}. Recently, there are also
quadratures based on the harmonic density interpolation~\cite{perez2019harmonic},
planewave density interpolation~\cite{perez2019planewave},
and adaptive integration~\cite{greengard2021jcp} that combines the FMM
and careful precomputed tables.

In this paper, we introduce a unified quadrature scheme that handles self and near
interactions, as well as close evaluations, eliminating the need for adaptive integration.
For the Laplace double layer potential (DLP), 
we first 
approximates the density $\sigma$ on $P$ in {\it quaternion algebra} such that the
resulting approximation has a {\it natural harmonic polynomial extension}.
The harmonic polynomial extension is valid in the whole ambient space,
but only needed in a {\it volumetric neighborhood} of $P$. Such density
approximation enables the application of the {\it general Stokes theorem}
without any further
approximation in exact arithmetics and reduces singular and nearly
singular integrals on the surface patch $P$
to line integrals on the boundary of $P$. In other words, the scheme approximates
the density in such a way that one may apply the integration by parts formula to
analytically reduce surface integrals to line integrals on the boundary of each
surface patch. This differs from all prior approaches except the work
in \cite{zhu2021thesis,zhu2022sisc}.
We then derive new expressions of these line integrals such that they preserve
the singularity and translation invariance of the kernels in the original surface integrals.
The advantage of doing this is that we can apply the singularity-swap
quadrature~\cite{afklinteberg2021bit} 
to evaluate these nearly singular line integrals analytically by recurrence.
To be more precise, the new line integrals can be evaluated via another
step of integration by parts, this time by extending
the parameter into the complex plane rather than using the general Stokes theorem.
Hence the {\it recursive reduction quadrature} (RRQ).

We describe in detail the improvements made in RRQ over the work
in \cite{zhu2021thesis,zhu2022sisc} in \cref{sec:improvements}. 
Here, we summarize the novel features that make it more elegant, accurate, and efficient
than other existing methods.

\vspace{.2in}

\begin{enumerate}[label=(\arabic*)]
\item Unlike the popular Duffy quadrature used in boundary element methods
  and its aforementioned descendants, which are used
  only for self-interactions, RRQ unifies the treatment of self and near interactions,
  as well as close evaluations.
\item Unlike other existing methods that evaluate the original surface integrals directly
  using some kind of adaptive integration, oversampling,
  target interpolation or extrapolation,
  RRQ reduces surface integrals via integration-by-parts twice, achieving full reduction
  from surface integrals to function
  evaluations at the vertices of the surface elements.
\item RRQ further divides target points into two groups. The first group consists of
  points close to the surface but sufficiently away from the edges of the surface
  element, where much less computational effort is needed to evaluate the line
  integrals. The second group includes points close to the edges of the surface
  element, where the singularity-swap quadrature requires only a mild increase in
  computational effort. Thus, RRQ reduces the list of targets that require
  significant computational effort from a neighborhood of the entire surface
  element to the neighborhood of its boundary.
\item Unlike some existing methods that rely on asymptotics or low-order density
  approximations, RRQ can achieve much higher convergence order and accuracy.
\end{enumerate}

\vspace{.2in}

Combining all these features, RRQ achieves high efficiency and robustness and can
be constructed to have arbitrarily high order. It attains twelve-digit accuracy uniformly
for close evaluations and provides a speedup of fivefold or more over previous
state-of-the-art quadrature methods.

The outline of the paper is as follows. In \cref{sec:prelim}, we collect and review
the analytical apparatus used in the rest of the paper. In \cref{sec:mainresults},
we present the main theoretical results. In \cref{sec:algorithm}, we discuss the
algorithmic steps in the scheme and analyse the cost of each step. Numerical examples
are shown in \cref{sec:numericalexamples}. Finally, \cref{sec:conclusions} contains
some further remarks on future extensions of the work.

\section{Mathematical apparatus}\label{sec:prelim}
In this section, we collect some results to be used in later sections, including
quaternion algebra, exterior derivatives, differential forms, general Stokes theorem,
harmonic polynomials, singularity-swap quadrature, and key results from
\cite{zhu2021thesis,zhu2022sisc}. We would like to remark here that even though it
might be possible to write the entire paper using vector calculus, we choose to
use quaternion algebra and differential forms 
due to the following reasons.
First, quaternion algebra greatly simplifies many equations.
Second, exterior derivatives and differential forms are standard tools
for the general Stokes theorem.
The use of differential forms, which employ Cartesian coordinates in the ambient space,
also makes it easy to understand why close evaluation can be treated the
same way as near interactions.
Nevertheless, the paper remains mostly self-contained, and the use of quaternion
algebra and differential forms is minimized to ensure it is easily accessible to readers.

\subsection{Quaternion algebra}
The quaternion algebra, as pointed out in \cite{rosen2019}, is the even part of the Clifford
algebra. For a quaternion $f=f_0+f_1\bi +f_2\bj+f_3\bk$, we also write it
as $f=(f_0,\f)$, where $f_0$ and $\f$ are the scalar and vector parts of the
quaternion $f$, respectively. Here, we identify three imaginary basis vectors
$\bi$, $\bj$, $\bk$ in the quaternion algebra with three basis vectors
in $\mathbb{R}^3$. The product of two quaternions, $f=(f_0,\f)$ and $g=(g_0,\g)$,
can then be expressed using the dot and cross products of vectors as follows:
\be\label{qprule}
fg=(f_0g_0-\f\cdot\g,f_0\g+g_0\f+\f\times\g),
\ee
where $\cdot$ denotes the vector dot product, and $\times$ denotes the cross product. 
The quaternion (and the Clifford) algebra is associative, but non-commutative. That is,
if $f$, $g$, and $h$ are three quaternions, then $(fg)h$ = $f(gh)$, but $fg\ne gf$ in general.
The conjugate of a quaternion $f$ is denoted as $\bar{f}=(f_0,-\f)$, and the length of $f$
is denoted as $|f|=\sqrt{f\bar{f}}=\sqrt{f_0^2+|\f|^2}$.
It is easy to check that $f_0 g=(f_0,\bzero)g=g(f_0,\bzero)$. Thus, we identify
a scalar $f_0$ with the quaternion $(f_0,\bzero)$.

\subsection{Differential forms and Stokes theorem}
The general Stokes theorem on differential forms is stated as \cite{docarmo1998}
\be\label{stokestheorem}
\int_{\partial D}\omega = \int_D d\omega,
\ee
where $\omega$ is a $k$-form, $d$ is the exterior derivative, and $\partial D$ is the
boundary of $D$. A differential form $\omega$
is closed if $d\omega=0$, and exact if $\omega=d\alpha$
for some other differential form $\alpha$. Because $d^2=0$, every exact form is closed.
The Poincar{\' e} lemma \cite{docarmo1998}
states that every closed form is exact when the domain is contractible. For the rest of
the paper, we always assume that the surface $S$ is divided into a collection of
contractible patches such that the Poincar{\' e} lemma is valid on every patch $P$.
It is clear that the assumption holds for all reasonable discretizations encountered
in practice.
\Rd
The vector area differential $\bnu(\br)da(\br)$ can be expressed in terms of differential 2-forms as follows:
\be\label{areadifferential}
\bnu(\br)da(\br) = dy\wedge dz\,\bi + dz\wedge dx\,\bj + dx\wedge dy\,\bk,
\ee
where $\br=(x,y,z)$. To verify this, consider a patch with the parametrization $\br(\theta,\phi)$. Standard vector calculus identities for the unit normal $\bnu(\br)$ and the scalar area element $da(\br)$ yield
\be
\bnu(\br)da(\br) = \left(\frac{\partial \br}{\partial \theta} \times \frac{\partial \br}{\partial \phi} \right) d\theta \wedge d\phi.
\ee
On the other hand, expanding the first term on the right-hand side of \eqref{areadifferential} gives
\be
\ba
dy\wedge dz\,\bi &= \left(\frac{\partial y}{\partial \theta} d\theta
+ \frac{\partial y}{\partial \phi} d\phi \right) \wedge \left(\frac{\partial z}{\partial \theta} d\theta+ \frac{\partial z}{\partial \phi}
d\phi \right) \bi \\
&= \left(\frac{\partial y}{\partial \theta} \frac{\partial z}{\partial \phi} -
\frac{\partial y}{\partial \phi} \frac{\partial z}{\partial \theta} \right)\bi \, d\theta\wedge d\phi\\
&= \left(\frac{\partial \br}{\partial \theta} \times \frac{\partial \br}{\partial \phi} \right)_1 \bi d\theta \wedge d\phi.
\ea
\ee
Analogous calculations for the other two terms confirm the identity in \cref{areadifferential}.
\Bk
Thus, a surface integral of a vector field in three dimensions
is equivalent to the integral of differential 2-forms, expressed as follows:
\be\label{surfintegral2form}
\int_S \f(\br)\cdot \bnu(\br)da(\br)= \int_S f_1dy\wedge dz+f_2dz\wedge dx+f_3dx\wedge dy.
\ee
The following lemma summarizes the analytical results required to convert
surface integrals over a patch $P$ to line integrals along its boundary $\partial P$.
\begin{lemma}\label{lemma1}
Let $P$ be a contractible, oriented surface in $\mathbb{R}^3$, and let $\f$ be a vector field defined in a neighborhood of $P$. The differential 2-form $\alpha = \f(\br) \cdot \bnu(\br) \, da(\br)$ is closed if and only if $\nabla \cdot \f = 0$. Furthermore, if $\alpha$ is closed, then for any fixed point $\bp \in \mathbb{R}^3$, the 1-form $\beta$ defined by
\be\label{1formconstruction}
\beta = \left( -\int_0^1 t(\br-\bp) \times \f(\bp+t(\br-\bp)) \, dt \right) \cdot d\br
\ee
satisfies the equation $d\beta = \alpha$. Consequently,
\be
\int_P \alpha = \int_{\partial P} \beta.
\ee
In other words, the surface integral $\int_P \alpha$ reduces to a line integral over the boundary $\partial P$, with the 1-form $\beta$ given explicitly by \cref{1formconstruction}.
\end{lemma}
\begin{proof}
By \cref{surfintegral2form}, the exterior derivative of $\alpha$ is given by
\be
d\alpha = (\nabla\cdot \f) \, dx\wedge dy\wedge dz,
\ee
which establishes the first assertion immediately.

\Rd
To verify that the 1-form $\beta$ defined in \eqref{1formconstruction} satisfies $d\beta = \alpha$, recall that for any vector field $\bv$, the exterior derivative $d(\bv\cdot d\br)$ corresponds to the curl of $\bv$:
\be\label{1formconstruction2}
d(\bv\cdot d\br) = (\nabla \times \bv)_1 \, dy\wedge dz +
(\nabla \times \bv)_2 \, dz\wedge dx +
(\nabla \times \bv)_3 \, dx\wedge dy.
\ee
We apply the standard vector calculus identity
\be
\nabla \times (\bA\times\bB) = (\bB\cdot\nabla)\bA - (\bA\cdot\nabla)\bB + \bA(\nabla\cdot\bB) - \bB(\nabla\cdot\bA)
\ee
to the term inside the integral of \cref{1formconstruction}. Letting $\bA = \br - \bp$, and noting that $(\f \cdot \nabla)\bA = \f$ and $\nabla \cdot \bA = 3$, we obtain
\be\label{1formconstruction3}
\nabla \times \left((\br-\bp)\times \f\right) =
-2\f -((\br-\bp)\cdot\nabla)\f.
\ee
Here, we also used the assumption that $\nabla\cdot\f = 0$.

Let $d\beta = P \, dy\wedge dz + Q \, dz\wedge dx + R \, dx\wedge dy$. Substituting \cref{1formconstruction3} into \cref{1formconstruction2} and taking into account the minus sign in \cref{1formconstruction}, the coefficient $P$ is given by
\be
\ba
P &= \int_0^1 -\left(\nabla\times\left(t(\br-\bp)\times \f(\bp+t(\br-\bp))\right)\right)_1 dt \\
&= \int_0^1  2t f_1(\bp+t(\br-\bp)) + t^2 ((\br-\bp)\cdot\nabla)f_1(\bp+t(\br-\bp))  dt \\
&= \int_0^1 \frac{d}{dt} \left( t^2 f_1(\bp+t(\br-\bp)) \right) dt \\
&= \left[ t^2 f_1(\bp+t(\br-\bp)) \right]_0^1 = f_1(\br).
\ea
\ee
Similar calculations for $Q$ and $R$ yield $f_2(\br)$ and $f_3(\br)$, respectively. Thus, $d\beta = \alpha$. The final conclusion then follows directly from the general Stokes theorem \cref{stokestheorem}.
\Bk
\end{proof}

\subsection{Harmonic polynomials}
The associated Legendre polynomials $\plm$ are defined by the formula~\cite{nisthandbook}
\be
\plm(x) = \frac{1}{2^l l!}(1-x^2)^{m/2}\frac{d^{l+m}}{dx^{l+m}}(x^2-1)^l.
\ee
Here, $l$ is a nonnegative integer and $-l\le m \le l$. For a given $x$,
all values of $\plm(x)$ for $l=0,\ldots,p$ and $m=0,\ldots,l$ can be
calculated via the following recurrence formulas:
\be
\ba
P_0^0(x)&=1,\\
P_{l+1}^{l+1}(x)&=-(2l+1)\sqrt{1-x^2}P_l^l(x),\\
P_{l+1}^l(x)&=(2l+1)xP_l^l(x),\\
(l-m+1)P_{l+1}^m(x)&=(2l+1)xP_l^m(x)-(l+m)P_{l-1}^m(x).
\ea\label{associatelegendrerecurrence}
\ee

The spherical harmonic polynomials $\ylm$ are defined by the formula
\be\label{ylmdef}
\ylm(\theta,\phi)=\sqrt{2l+1}\sqrt{\frac{(l-m)!}{(l+m)!}}\plm(\cos(\theta))e^{im\phi},
\ee
where we have dropped the normalization factor $\frac{1}{\sqrt{4\pi}}$ and the
Condon-Shortley phase $(-1)^m$.

The regular solid harmonic polynomials $\rlm$ are defined by the formula
\be\label{rlmdef}
\rlm(\br) = \frac{1}{\sqrt{2l+1}}r^l \ylm(\theta,\phi),
\ee
where $(r,\theta,\phi)$ are the spherical polar coordinates of $\br$.
Note that $\rlm(\br)$ is a homogeneous polynomial of degree $l$ in its
Cartesian coordinates $(x,y,z)$.

The translation of the regular solid harmonic is given by the formula~\cite{fmm2,greengard1987thesis,greengard1988}
\be\label{l2ltranslation}
\rlm(\br)=\sum_{j=0}^l \sum_{k=-j}^j
R_j^k(\br-\brp)R_{l-j}^{m-k}(\brp)T_{jk,lm},
\ee
where the coefficient $T_{jk,lm}$ is given by the formula
\be\label{tjklmdef}
T_{jk,lm}={\binom{l+m}{j+k}}^{1/2} {\binom{l-m}{j-k}}^{1/2}.
\ee

Harmonic functions can be approximated by a harmonic polynomial expansion of the
$p$th order
\[\sum_{l=0}^p\sum_{m=-l}^l C_l^m \rlm(\br).\]
Note that this is exactly the so-called local expansion in the fast multipole method~\cite{fmm2},
and it can be evaluated via the recurrence \cref{associatelegendrerecurrence}
efficiently.
The translation of local expansions
can be carried out using \cref{l2ltranslation} and \cref{tjklmdef} as follows:
\be\label{l2ltranslation2}
\sum_{l=0}^p\sum_{m=-l}^l C_l^m \rlm(\br)=
\sum_{j=0}^p\sum_{k=-j}^j D_j^k R_j^k(\br-\brp),
\ee
where
\be\label{l2ltranslation3}
D_j^k=\sum_{l=j}^p\sum_{m=-l}^l 
T_{jk,lm} R_{l-j}^{m-k}(\brp)C_l^m.
\ee

\begin{remark}\label{rem:fasttranslation}
The direct application of \cref{l2ltranslation2}--\cref{l2ltranslation3} requires $O(p^4)$
work. In \cite{fmm6}, the translation is carried out by rotation of $\br-\brp$ to $z$-axis,
translating along the $z$-axis, and another rotation back to the original coordinate system.
The total cost of the scheme in \cite{fmm6} is $O(p^3)$.
\end{remark}

\subsection{Singularity-swap quadrature}\label{sec:ssq}
In \cite{afklinteberg2021bit}, a singularity-swap quadrature is proposed to evaluate
nearly singular line integrals of the form
\be\label{lineintegral0}
I(\brp) = \int_{-1}^1 \frac{f(t)}{|\br(t)-\brp|^m}dt, \quad m=1,3,\ldots,
\ee
where $f$ is a smooth function and
$\br(t)=(x(t),y(t),z(t))$ is the vector function with the real-valued parameter $t$
defining a space curve $\Gamma$ in three dimensions.
The key observation in \cite{afklinteberg2021bit} is to find
a complex number $t_0$ such that
\be\label{complexparameter}
(x(t_0)-x')^2+(y(t_0)-y')^2+(z(t_0)-z')^2=0.
\ee
That is, we extend the real-valued argument of the function $\br(t)$ to the complex plane
and then find a complex number $t_0$ such that
\cref{complexparameter} holds
for the given target point $\brp=(x',y',z')\in \mathbb{R}^3$.
Then \cref{lineintegral0} can be rewritten as
\be\label{lineintegral1}
I(\brp) = \int_{-1}^1 \frac{F(t)}{((t-t_0)(t-\overline{t_0}))^{m/2}}dt,
\ee
where
\be
F(t)=f(t) \frac{|t-t_0|^m}{|\br(t)-\brp|^m}.
\ee

Since both $|t-t_0|^2=(t-t_0)(t-\overline{t_0})$
and $|\br(t)-\brp|^2=(x(t)-x')^2+(y(t)-y')^2+(z(t)-z')^2$
have the same roots $t_0$ and $\overline{t_0}$ in the complex plane,
$F(t)$ is again a smooth function of $t$ for $t\in [-1,1]$. And the construction
of the quadrature is completed by the observation that
integrals of the form
\[
\int_{-1}^1 \frac{t^i}{((t-t_0)(t-\overline{t_0}))^{m/2}}dt, \quad i=0,1,2,\ldots \]
can be evaluated analytically via a recurrence relation.

\subsection{Summary of main analytic results in \cite{zhu2021thesis,zhu2022sisc}}
\label{sec:dlpsummary}
At the heart of the quadrature scheme in~\cite{zhu2021thesis,zhu2022sisc}
lies the {\it local quaternion harmonic approximation} of the density on each patch $P$.
Consider the Laplace double layer potential (DLP) defined by \cref{dlpdef} with the whole
surface $S$ replaced by a \Bl contractible \Bk patch $P$.
Suppose that $H^{(l,m)}(\br)$ are harmonic polynomials. Then 
the following quaternion 2-forms~\cite[Corollary 3.4]{zhu2022sisc}
\be\label{dlp2form}
\dlptform(\brp,\br) = (0, \nabla G(\brp,\br))
(0,\bnu(\br)) 
(0,\nhlm(\br))da(\br)
\ee
are closed and exact.
Thus, there exist 1-forms $\dlpoform$ such that
\be\label{dlp1form}
d\dlpoform(\brp,\br) = \dlptform(\brp,\br).
\ee
By direct calculation, we have
\be
(0, \nabla G(\brp,\br))(0,\bnu(\br)) (\mu(\br),\bzero)
=(-\left[\nabla G(\brp,\br)\cdot \bnu(\br)\right]\mu(\br),
\left[\nabla G(\brp,\br)\times \bnu(\br)\right]\mu(\br)),
\ee
and the following lemma.
\begin{lemma}\label{lem:dlpreduction}
  Suppose that the density $\mu$ of the Laplace double layer potential
  is approximated via the following $p$th order quaternion harmonic approximation
  on a \Bl contractible \Bk patch $P$~\cite[Eq. (21)]{zhu2022sisc} 
\be\label{dlpdensappr}
(\mu(\br),\bzero) \approx \sum_{l=1}^p\sum_{m=1}^l
(0, \nabla H^{(l,m)}(\br)) c^{(l,m)}, \quad \br \in P,
\ee
where $c^{(l,m)}$ are
quaternion coefficients of the form
\be
c^{(l,m)}=c_0^{(l,m)}+ c_1^{(l,m)}\bi+ c_2^{(l,m)}\bj+ c_3^{(l,m)}\bk.
\ee
Suppose further that the quaternion $Q$ is defined by the formula
\be
Q[\mu](\brp) = \int_P (0, \nabla G(\brp,\br))(0,\bnu(\br)) (\mu(\br),\bzero) da(\br).
\ee
Then
\be
\dlp[\mu](\brp) = -Q_0[\mu](\brp),
\ee
i.e., the double layer potential is the negative scalar part of the quaternion $Q$.
Thus,
\be
\ba
\dlp[\mu](\brp) &= -Q_0[\mu](\brp)\\
&\approx -\sum_{l,m} \int_P\left[\dlptform(\brp,\br) \clm\right]_0\\
& = -\sum_{l,m} \int_{\partial P} \left[\dlpoform(\brp,\br) \clm\right]_0,
\ea\label{dlp1formrep}
\ee
where $[\cdot]_0$ denotes the scalar part of the quaternion.
\end{lemma}

\Cref{lem:dlpreduction} shows that 
the evaluation of the Laplace double layer potential on $P$
is reduced to that of the quaternion 1-forms (line integrals) on $\partial P$.

\section{Main results}
\label{sec:mainresults}

\subsection{Summary of improvements made over the work
  in \cite{zhu2021thesis,zhu2022sisc}}\label{sec:improvements} 
The work in \cite{zhu2021thesis,zhu2022sisc} is quite original and offers a fresh
starting point for the evaluation of layer potentials in three dimensions.
However, it is incomplete and has several major drawbacks.
First, \cite{zhu2021thesis,zhu2022sisc} provides explicit expressions of the harmonic
basis functions only up to order seven. A rigorous proof that those functions
form a complete basis is missing. It is also unclear how to construct
such basis functions for arbitrary order. Second, the kernels of line integrals
in \cite{zhu2021thesis,zhu2022sisc} have weaker but more complicated singularity.
Indeed, the construction of line integrals in \cite{zhu2021thesis,zhu2022sisc} depends
on the choice of the origin, and as a result the singularity of the kernels depends
not only on the difference vector of the target and source points, but also on the
absolute coordinates of the target and source points.
In other words, the translation invariance of the kernel is lost. This leads to
expensive adaptive integration to resolve the near-singularity in the kernel when
the target point is close to the boundary curves.
Consequently, the efficiency gains achieved by the reduction to line integrals
are significantly eroded, resulting in unsatisfactory performance.
Third, in \cite{zhu2021thesis,zhu2022sisc} the Laplace single
layer potential (SLP) is converted to the Laplace double layer potential with a
{\it target-dependent density}. While this method works, it is somewhat {\it ad hoc}
and requires three quaternion DLP density approximations, resulting in significant
computational slowdown.

Here, we fill in the missing pieces and present solutions to overcome these drawbacks,
which significantly improves performance and reliability.
First, we show that a set of complete and linearly independent basis functions
can be chosen for the quaternion harmonic polynomial approximation of the density
for flat triangular patches for any given order (\cref{thm:trianglebasis}).
This puts the construction of the quadrature scheme in a solid footing and
extends the scheme to arbitrary high order.
Second, we develop a new integration-by-parts formula for the Laplace single
layer potential (\cref{thm:slpreduction}), resulting in a threefold speedup.
Third, in order to remove the adaptive
integration from the scheme, we switch to the target-centered 1-forms
(\cref{thm:dlp1form} and \cref{thm:cslp1form}) (i.e., line
integrals) from the origin-centered 1-forms in \cite{zhu2021thesis,zhu2022sisc}.
These target-centered 1-forms preserve the singularity of the kernel, enabling
the use of singularity-swap quadrature~\cite{afklinteberg2021bit} 
to evaluate these nearly singular line integrals analytically by recurrence.
To be more precise, the 1-forms are further reduced to 0-forms
through another application of integration by parts, achieved
by extending
the parameter into the complex plane rather than using the general Stokes theorem.
Finally, we apply the fast translation~\cite{fmm6} of local expansions in the FMM and
carefully separate the source-dependent steps from the target-dependent steps
in the whole scheme to improve its efficiency (\cref{sec:algorithm}).

\subsection{Choice of basis functions for the $p$th order harmonic approximation}
\label{sec:completebasis}
We now discuss how to choose basis functions $\hlm$ in the quaternion
harmonic approximation~\cref{dlpdensappr} for the $p$th order scheme.
For standard polynomial approximation on a flat triangle on
the $xy$-plane, one may use $x^iy^j$ with $0\le i+j \le p-1$ or their suitable linear
combinations (such as the Koornwinder polynomials~\cite{koornwinder1975}) as the basis functions
for the $p$th order scheme. It is not clear how one should choose harmonic polynomials
$\hlm$ ($1\le m \le l \le p$) in 
\cref{dlpdensappr} such that their gradients 
span the proper subspace for the desired convergence order. 
In \cite{zhu2021thesis,zhu2022sisc}, the authors explicitly describe basis functions $\hlm$ 
for $p\le 7$ and observe that empirically it works. 
Below we show how to select such basis functions
$\hlm$ ($1\le m \le l \le p$) for any given order $p$ for flat triangles on
the $xy$-plane.

We start from the standard solid harmonic polynomials $R_l^m(\br)=R_l^m(x,y,z)$ in
\cref{rlmdef}. In order to obtain real harmonic polynomials, we consider
\be\label{htlmdef}
\htlm(x,y,z) = \frac{i}{\sqrt{2}}\left(-(-1)^m R_l^m(x,y,z)+R_l^{-m}(x,y,z)\right).
\ee
In Cartesian coordinates, we have~\cite{nisthandbook}
\be
\htlm(x,y,z)= \left[\frac{2(l-m)!}{(l+m)!}\right]^{1/2}\Pi_l^m(z,r)B_m(x,y),
\ee
where
\be
B_m(x,y)=\frac{1}{2i}[(x+iy)^m-(x-iy)^m],
\ee
and
\be
\Pi_l^m(z,r)=r^{l-m}\frac{d^mP_l(u)}{du^m}=\sum_{k=0}^{\lfloor(l-m)/2\rfloor}\gamma_{lk}^{(m)}r^{2k}z^{l-2k-m},
\ee
with $P_l$ the Legendre polynomial of degree $l$, $u=\cos\theta=\frac{z}{r}$, and
\be
\gamma_{lk}^{(m)}=(-1)^k2^{-l} {\binom{l}{k}}{\binom{2l-2k}{l}}\frac{(l-2k)!}{(l-2k-m)!}.
\ee

The following theorem provides one way of constructing proper basis functions $\hlm$
($1\le m \le l \le p$) for the quaternion harmonic approximation \cref{dlpdensappr}.
\begin{theorem}\label{thm:trianglebasis}
  Let
  \be\label{hlmdef}
  H^{(l,m)}(x,y,z)\coloneqq \htlm(y,z,x) = \left[\frac{2(l-m)!}{(l+m)!}\right]^{1/2}\Pi_l^m(x,r)B_m(y,z).
  \ee
  Then on the $xy$-plane,
  \be
  \nabla \hlm(x,y,0)=(0,0,\frac{\partial \hlm(x,y,0)}{\partial z})\coloneqq (0,0,\Flm(x,y)),
  \ee
  and for any $p\ge 1$,
  \be
  \Span\{\Flm | l=1,\ldots,p, \quad m=1,\ldots,l\}=\Span\{x^iy^j | 0\le i+j\le p-1\}.
  \ee
  In other words, $\Flm$ ($l=1,\ldots,p$, $m=1,\ldots,l$) form a complete and linearly independent
  basis for the $p$th order polynomial approximation on the $xy$--plane.
\end{theorem}

\Cref{thm:trianglebasis} is proved by induction and direct calculations. We break the proof
into the following lemmas.

\begin{lemma}
  \be\label{bfactor}
  B_m(y,z)=zC_{m-1}(y,z),
  \ee
  where $C_{m-1}(y,z)$ is a polynomial in $y$ and $z$ of degree at most $m-1$.
\end{lemma}

\begin{proof}
  By the definition of $B_m$, we have
  \be
  B_m(y,z)=\frac{1}{2i}[(y+iz)^m-(y-iz)^m].
  \ee
  Thus, $B_m(y,z)$ is a polynomial in $y$ and $z$ of degree at most $m$. Furthermore, it is clear
  that $B_m(y,0)=0$. That is, $z=0$ is a root of $B_m(y,\cdot)$ as a function of $z$. Therefore,
  $z$ can be factored out, which leads to \eqref{bfactor}.
\end{proof}

\begin{lemma}
  \be\label{hxy}
  \frac{\partial \hlm(x,y,0)}{\partial x}=0, \quad   \frac{\partial \hlm(x,y,0)}{\partial y}=0.
  \ee
\end{lemma}
\begin{proof}
  By direct calculations, we have
  \be
  \ba
  \frac{\partial \hlm(x,y,z)}{\partial x}&= B_m(y,z)\frac{\partial \Pi_l^m(x,y,z)}{\partial x},\\
  \frac{\partial \hlm(x,y,z)}{\partial y}&= B_m(y,z)\frac{\partial \Pi_l^m(x,y,z)}{\partial y}+z\frac{\partial C_{m-1}(y,z)}{\partial y}.
  \ea
  \ee
 And the lemma follows from $B_m(y,0)=0$. 
\end{proof}

Direct calculations also lead to the following lemmas.
\begin{lemma}
  \be\label{pixy0}
  \Pi_l^m(x,r|_{z=0})\coloneqq D_l^m(x,y)=\sum_{k=0}^{\lfloor(l-m)/2\rfloor}\gamma_{lk}^{(m)}(x^2+y^2)^{k}x^{l-2k-m}.
  \ee
  In particular,
  \be
  D_l^l(x,y)=\gamma_{l0}^{(l)}, \quad D_l^{l-1}(x,y) = \gamma_{l0}^{(l-1)}x.
  \ee
\end{lemma}

\begin{lemma}
  \be\label{hz}
  \frac{\partial \hlm(x,y,0)}{\partial z}=\Flm(x,y)=D_l^m(x,y) C_{m-1}(y,0).
  \ee
\end{lemma}

\begin{lemma}
  \be\label{cy0}
  C_{m-1}(y,0)=my^{m-1}, \quad m=1,2,\ldots
  \ee
\end{lemma}

\begin{lemma}\label{lem:flhomo}
  For $l\ge 1$, $\Flm$ is a homogeneous polynomial of degree $l-1$ for all $m=1,\ldots,l$.
  Furthermore for any $0\le k\le l-1$,
  \be\label{flmspan}
  \Span\{\Flm(x,y)|m=l,l-1,\ldots,l-k\}=\Span\{y^{l-1},xy^{l-2},\cdots, x^ky^{l-k-1}\}.
  \ee
\end{lemma}

\begin{proof}
  We prove it by induction. First, the statement is true for $k=0$ since 
  \be
  F^{(l,l)}(x,y)=D_l^l(x,y)C_{l-1}(y,0)=l\gamma_{l0}^{(l)} y^{l-1}.
  \ee
  Suppose that the statement is true for all $k \le n$. Then for $k=n+1$, we have
  \be\label{fspan}
  \Span\{\Flm(x,y)|m=l,\ldots,l-n-1\}=\Span\{y^{l-1},\cdots, x^ny^{l-n-1}, F^{(l,l-n-1)}(x,y)\}.
  \ee
  Combining \eqref{pixy0}--\eqref{cy0}, we have
  \be\label{fln1}
  \ba
  F^{(l,l-n-1)}(x,y)&=(l-n-1)y^{l-n-2}\sum_{k=0}^{\lfloor(n+1)/2\rfloor}\gamma_{lk}^{(l-n-1)}(x^2+y^2)^{k}x^{n+1-2k}\\
  &=c_0y^{l-n-2}x^{n+1}+c_1y^{l-n}x^{n-1}+\ldots.
  \ea
  \ee
  It is easy to see that all lower order terms in \eqref{fln1} already appear in the first $n$ terms
  of the right side of \eqref{fspan}.
  Thus,
  \be
  \Span\{y^{l-1},\cdots, x^ny^{l-n-1}, F^{(l,l-n-1)}(x,y)\}=
  \Span\{y^{l-1},\cdots, x^ny^{l-n-1}, x^{n+1}y^{l-n-2}\}.
  \ee
  And the lemma follows.
\end{proof}

It is clear that \cref{thm:trianglebasis} follows from \cref{lem:flhomo}.

We now choose a set of collocation points $\br^{(i,j)}$ ($0\le i+j \le p-1$)
on the triangular patch $P$. In the community of integral equation methods, popular
choices include Vioreanu-Rokhlin nodes~\cite{vioreanu2014sisc} and
Xiao-Gimbutas nodes~\cite{xiao2010cma}, with the latter having associated
quadrature close to the Gaussian quadrature on a triangle. 
The quaternion coefficients $c^{(l,m)}$ in \cref{dlpdensappr}
are then obtained by solving the following
linear system, which imposes equality at these points in \cref{dlpdensappr}):
\be\label{dlpdensitylinearsystem}
\sum_{l=1}^p\sum_{m=1}^l
(0, \nabla H^{(l,m)}(\br^{(i,j)})) c^{(l,m)} = (\mu(\br^{(i,j)}),\bzero), \quad 0\le i+j \le p-1.
\ee

Writing out explicitly, \cref{dlpdensitylinearsystem}
is equivalent to the following equations with only the usual vector calculus involved:
\be\label{dlpdensitylinearsystem1}
\ba
-\sum_{l,m} \nhlm(\br^{(i,j)})\cdot \bclm &= \mu(\br^{(i,j)}), \\
\sum_{l,m} \clm_0\nhlm(\br^{(i,j)})+\nhlm(\br^{(i,j)})\times \bclm &= \bzero.
\ea
\ee
Let $\nabla H = (F_1, F_2, F_3)$. Then in block matrix form, we have
\be\label{dlpdensitylinearsystem2}
\begin{pmatrix}
  0 & -F_1 & -F_2 & -F_3\\
  F_1 & 0 & -F_3 & F_2\\
  F_2 & F_3 & 0 & -F_1\\
  F_3 & -F_2 & F_1 & 0
\end{pmatrix}
\begin{pmatrix} c_0\\c_1\\c_2\\c_3
\end{pmatrix}
= \begin{pmatrix} \mu\\0\\0\\0
\end{pmatrix},
\ee
where the length of $c_k$ ($k=0,1,2,3$)
is $\np$ with $\np=\frac{p(p+1)}{2}$,
the size of each block is
$\np  \times \np$, 
and the size of the full linear system
is $4\np \times 4\np$.

\begin{remark}
  \Cref{thm:trianglebasis} shows that when the patch
  is flat and lie in the $xy$--plane, the harmonic basis functions constructed in \cref{hlmdef}
  lead to $(F_1,F_2,F_3)=(0,0,F_3)$, and that $F_3$ is invertible for a set of collocation points
  $\{\br^{(i,j)} | 0\le i+j \le p-1\}$ if and only if the standard $p$th order monomial basis
  $\{x^iy^j | 0\le i+j\le p-1\}$ form linearly independent column vectors on those collocation
  points. Furthermore, since $(F_1,F_2,F_3)=(0,0,F_3)$, \cref{dlpdensitylinearsystem2} shows that
  the system matrix in the quaternion harmonic approximation is invertible if and only if $F_3$
  is invertible. This is due to the fact that in this case, the nonzero blocks in
  \cref{dlpdensitylinearsystem2} are $\pm F_3$ on the anti-diagonal.
\end{remark}

\begin{remark}
  For flat triangles in the $xy$-plane, $\nhlm \cdot \bnu = \Flm(x,y)$ ($1\le m \le l \le p$)
  span the same subspace
  as the  standard $p$th order monomial basis
  $\{x^iy^j | 0\le i+j\le p-1\}$. This is used in the single layer potential
  density approximation~\cref{slpdensappr}.
\end{remark}

For an arbitrary nonflat triangular patch in three dimensions
it is rather difficult to construct a set of complete and linearly independent
harmonic polynomial basis so that the resulting linear system \cref{dlpdensitylinearsystem}
is always solvable. However, all surface patches 
become closer and closer to being flat
when the patch size approaches zero. Thus, the harmonic basis functions that work for flat
triangles seem to be a good choice even for a general curvilinear triangular patch.
In practice, we observe that such basis functions 
lead to an invertible linear system in \cref{dlpdensitylinearsystem} 
with the desired convergence order $p$, even when the patch is relatively large and nonflat.

\subsection{New treatment of the Laplace single layer potential}
In this section, we present
a more natural reduction of the Laplace single layer potential (SLP)
to 1-forms, which requires
a scalar harmonic density approximation and a single DLP density approximation.
Since the cost of quaternion approximation and subsequent calculations is much more than
that of the scalar part, the cost of evaluating the SLP by our scheme
is close to that of the DLP. For comparison, the scheme in \cite{zhu2021thesis,zhu2022sisc}
requires three quaternion density approximations for the SLP evaluation.

\begin{lemma}\label{lemma2}
  Suppose that $u$ and $v$ are solutions to the Laplace equation.
  Suppose further that $P$ is a contractible patch on a smooth oriented surface $S$
  in three dimensions.
  Then the 2-form
  \be
  (u\nabla v-v\nabla u)\cdot \bnu da(\br)
  \ee
  is closed, and thus exact on $P$,
  where $da(\br)$ is the area differential. 
\end{lemma}

\begin{proof}
This simply follows from
  \be
  \nabla \cdot (u\nabla v-v\nabla u) = u\Delta v-v\Delta u = 0.
  \ee
\end{proof}

The following corollary is immediate.
\begin{corollary}\label{cor2}
  Suppose that $v$ is a harmonic function, i.e., $\Delta v=0$.
  Then the 2-form
  \Rd
  \be\label{slpquaternionform}
  \chi(\brp,\br) = (G(\brp,\br)\nabla v(\br)-v(\br)\nabla G(\brp,\br))\cdot \bnu(\br) da(\br)
  \ee
  \Bk
  is closed, i.e., $d\chi(\brp,\br) = 0$ for $\br\ne \brp$.
\end{corollary}

Substituting $v=\hlm$ with $\hlm$ defined in \cref{hlmdef} 
into \cref{slpquaternionform},
we obtain a closed and exact 2-form $\slptform$ defined by the formula
\be\label{slp2form}
\slptform(\brp,\br) = (G(\brp,\br)\nabla \hlm(\br) -\hlm(\br)\nabla G(\brp,\br))\cdot \bnu(\br)
da(\br).
\ee
We denote the 1-forms whose exterior derivative are $\slptform$ by $\slpoform$, that is,
\be\label{slp1form}
\slptform(\brp,\br) = d\slpoform(\brp,\br),
\ee
and by the Stokes theorem,
\be\label{slpformreduction}
\int_P \slptform(\brp,\br) = \int_{\partial P} \slpoform(\brp,\br).
\ee

\begin{theorem}\label{thm:slpreduction}
  Suppose that $P$ is a \Bl contractible \Bk patch in $\mathbb{R}^3$ and
  $\br^{(i,j)}$ are a set of $\np$ collocation points on $P$.
  Suppose further that
  the density in the Laplace single layer potential \cref{slpdef} is approximated
by
\be\label{slpdensappr}
\sigma(\br) \approx \sum_{l=1}^p\sum_{m=1}^l \nabla\hlm(\br)\cdot\bnu(\br) \dlm,
\ee
where $\dlm$ are {\it scalar} coefficients determined by the 
following 
$\np \times \np$
linear system:
\be\label{slpdensitylinearsystem}
\sum_{l=1}^p\sum_{m=1}^l
\nabla H^{(l,m)}(\br^{(i,j)})\cdot \bnu(\br^{(i,j)}) \dlm = \sigma(\br^{i,j}).
\ee
Let $\rho(\br)$ be an intermediate DLP density defined by the formula
\be\label{dlpdens}
\rho(\br)=\sum_{l,m}\hlm(\br)\dlm,
\ee
and approximate $\rho(\br)$ using quaternion basis $(0,\nhlm)$
(cf. \cref{dlpdensappr}):
\be\label{dlpdensappr2}
(\rho(\br),\bzero) \approx \sum_{l,m}(0,\nhlm(\br))\glm.
\ee
Then the Laplace single layer potential on $P$ can be reduced to the evaluation
of 1-forms as follows:
\be\label{slp1formrep}
\ba
\slp[\sigma](\brp)&=\int_P G(\brp,\br)\sigma(\br)da(\br)\\
&=\sum_{l,m}\int_{\partial P} \left(\slpoform(\brp,\br) \dlm - \left[\dlpoform(\brp,\br)\glm\right]_0\right),
\ea
\ee
where $\slpoform$ is defined in \cref{slp1form} and $\dlpoform$ is defined
in \cref{dlp1form}.
\end{theorem}

\begin{proof}
Combining \cref{slp2form} and \cref{slpdensappr}, we obtain
\be\label{slpappr}
\ba
G(\brp,\br)\sigma(\br)da(\br)
&\approx \sum_{l,m} G(\brp,\br)\nhlm(\br)\cdot \bnu(\br) \dlm da(\br)\\
&=\sum_{l,m}\slptform(\brp,\br) \dlm +
\nabla G(\brp,\br)\cdot\bnu(\br)
\rho(\br)
da(\br).
\ea
\ee
Note that the second term on the right hand side of \cref{slpappr}
is the Laplace double layer potential differential. 
Thus, using the quaternion approximation \cref{dlpdensappr2}
and \cref{lem:dlpreduction}, we have
\be\label{slpappr2}
(0, \nabla G(\brp,\br))(0,\bnu(\br)) \rho(\br)
da(\br) \approx \sum_{l,m}\dlptform(\brp,\br) \glm,
\ee
with exact quaternion 2-forms $\dlptform$ defined in \cref{dlp2form}.

Combining \cref{slpappr} -- \cref{slpappr2}, we obtain
\be\label{slp2formrep}
G(\brp,\br)
\sigma(\br)da(\br)
\approx \sum_{l,m}\slptform(\brp,\br) \dlm
- \sum_{l,m} \left[\dlptform(\brp,\br)\glm\right]_0.
\ee
Finally, combining \cref{dlp1form}, \cref{slp1form}, \cref{slpformreduction},
and \cref{slp2formrep}, we obtain \cref{slp1formrep}.
\end{proof}

In other words, the evaluation of the single layer potential is carried out by
a scalar approximation to the original density in \cref{slpdensappr}, 
a quaternion approximation \cref{dlpdensappr2} to the intermediate DLP density in
\cref{dlpdens}, and then reduction to line integrals by the general Stokes theorem.

\begin{remark}
  \Cref{slp1formrep} is similar to the treatment of the Laplace
  single layer potential in two dimensions in \cite{helsing2008jcp2,helsing2008jcp1},
  where one evaluates 
  the integrals involving the logarithmic kernel via the formula
\be
\int_{a}^b \log(\tau-z_t) \tau^{n-1} d\tau = \frac{1}{n} \log(\tau-z_t) \tau^n\Big|_a^b
- \frac{1}{n}\int_{a}^b\frac{\tau^n}{\tau-z_t} d\tau.
\ee
\end{remark}

\subsection{Singularity preserving 1-forms}
The differential 1-forms $\dlpoform$ in \cref{dlp1formrep} and $\slpoform$ 
in \cref{slp1formrep} are not unique. Indeed, different choice of the reference
point $\bp$ in \cref{1formconstruction} leads to different 1-forms.
In \cite{zhu2021thesis,zhu2022sisc}, $\bp=\bzero$ and the singularity in the resulting 1-forms
for the double layer potential is fairly complicated (see Appendix B of \cite{zhu2022sisc}
for explicit expressions). These nearly singular line integrals often require
expensive adaptive integration when the target point $\brp$ is close to the boundary
curves of $P$. We propose to use $\brp$ as the reference point when constructing
$\dlpoform$ and $\slpoform$. The advantage
of using $\bp=\brp$ in \cref{1formconstruction} is that the resulting 1-forms
have the same singularity as the Green's function and its derivatives, i.e., $1/|\brp-\br|^n$
for some positive integer $n$. 
The singularity-swap quadrature~\cite{afklinteberg2021bit} can then be used
to evaluate these
line integrals via analytical recurrence after polynomial approximation of the density
and a root-finding step. In other words, 1-forms are reduced to 0-forms and
the need of adaptive integration is removed.

We now derive these target-centered 1-forms. The translation of harmonic
polynomials becomes straightforward when complex solid harmonic polynomials are used.
Thus, we introduce
\be\label{slmdef}
\slm(\br)= \slm(x,y,z) = \rlm(y,z,x).
\ee
Combining \cref{htlmdef} and \cref{hlmdef}, we have
\be\label{hlmslmconnection}
\hlm(\br) = \sqrt{2} \Im(\slm(\br)).
\ee
The translation of $\slm$ is the same as that of $\rlm$ in \cref{l2ltranslation}, which leads to
\be\label{slmtranslation0}
\ba
\slm(\brp+t(\br-\brp)) &=\sum_{j=0}^l \sum_{k=-j}^j
S^{(j,k)}(t\br)S^{(l-j,m-k)}((1-t)\brp)T_{jk,lm}\\
&=\sum_{j=0}^l t^j(1-t)^{l-j} \sum_{k=-j}^j
S^{(j,k)}(\br)S^{(l-j,m-k)}(\brp)T_{jk,lm}.
\ea
\ee
Instead of real differential 2-forms $\dlptform$ in \cref{dlp2form} for the DLP,
we consider complex differential 2-forms
\be\label{cdlp2form}
\cdlptform(\brp,\br) = (0, \nabla G(\brp,\br))(0,\bnu(\br)) (0, \nslm(\br))da(\br).
\ee

\begin{theorem}[DLP 1-forms]\label{thm:dlp1form}
  Let $\cdlptform$ be the complex differential 2-forms defined
  in \cref{cdlp2form}. Let $P$ be a triangular patch on $S$.
  Suppose that the complex differential 1-forms
  $\cdlpoform$ are defined by
  \be\label{cdlp1formrep}
\cdlpoform(\brp,\br)=\frac{1}{4\pi} 
\sum_{j=2}^l C_{j,l} \sum_{k=-j}^j
\qjk(\brp,\br)
S^{(l-j,m-k)}(\brp)T_{jk,lm},
\ee
where 
\be
C_{j,l}=\int_0^1 t^{j-2}(1-t)^{l-j}dt = \frac{(j-2)!(l-j)!}{(l-1)!},
\ee
and the quaternion 1-forms $\qjk$ are given by the formula
\be\label{qjkdef}
\qjk(\brp,\br)=
-\left(0,d\left(\frac{\brp-\br}{|\brp-\br|}\right)\right)
\left(0,\nabla S^{(j,k)}(\br)\right).
\ee
Suppose further that
\be\label{gammalmdef}
\gammalm(\brp,\br) = \frac{1}{4\pi} \omega(\brp,\br) (0, \nabla \slm(\brp)),
\ee
where $\omega=(\omega_0,\bomega)$ is a quaternion 1-form defined by
\be\label{omega0def}
\omega_0(\brp,\br) = \frac{((\brp-\br)\times\bu)\cdot d\br}{|\brp-\br|(|\brp-\br|+\bu\cdot(\brp-\br))},
\ee
\be\label{bomegadef}
\bomega(\brp,\br) = -\frac{d\br}{|\brp-\br|},
\ee
with $\bu$ a fixed unit vector satisfying the property that
\Bl $(\brp-\br)\cdot\bu\ge 0$ for any $\br\in P$.\Bk
Then
\be\label{cdlp1form}
d\left(\cdlpoform(\brp,\br)+\gammalm(\brp,\br)\right) = \cdlptform(\brp,\br).
\ee
\end{theorem}
\begin{proof}
Straightforward calculation shows that
\be
\ba
G(\brp,\brp+t(\br-\brp)) &= \frac{1}{4\pi t|\brp-\br|},\\
\nabla G(\brp,\brp+t(\br-\brp)) &= \frac{\brp-\br}{4\pi t^2|\brp-\br|^3},
\ea\label{gfexpansion}
\ee
and
\be\label{cdlp2form1}
\ba
\cdlptform(\brp,\br)
&= \left(\left(\nabla G \times \nabla \slm\right)\cdot \bnu,
-(\nabla G\cdot \bnu) \nabla \slm\right.\\
&+\left.(\nabla G \times \bnu)\times \nabla \slm\right)da(\br),
\ea
\ee
Note that $1/t^2$ appears in $\nabla G(\brp,\brp+t(\br-\brp))$ because $G$ is singular
when $\br=\brp$. Thus, in order to apply \cref{1formconstruction} with $\bp=\brp$,
the density must be equal to zero at $t=0$. For this,
we split $\cdlptform$ into two terms:
\be\label{cdlp2form2}
\ba
\cdlptform(\brp,\br)
&=(0, \nabla G(\brp,\br))(0,\bnu(\br)) (0, \ntslm(\br))da(\br)\\
&+(0, \nabla G(\brp,\br))(0,\bnu(\br)) (0, \nslm(\brp))da(\br),
\ea
\ee
where
\Bl
\be
\ntslm(\br) = \nslm(\br)-\nslm(\brp),
\ee
\Bk
that is, we subtract the constant term so that $\ntslm(\brp)=\bzero$.
Now, \cref{slmtranslation0} leads to 
\be\label{slmtranslation1}
\ba
\ntslm(\brp+t(\br-\brp)) 
&=\sum_{j=2}^l t^{j-1}(1-t)^{l-j} \sum_{k=-j}^j
\nabla S^{(j,k)}(\br)S^{(l-j,m-k)}(\brp)T_{jk,lm}\\
&+((1-t)^{l-1}-1) \sum_{k=-1}^1
\nabla S^{(1,k)}S^{(l-1,m-k)}(\brp)T_{1k,lm},\\
\ea
\ee
where we have pulled out the $j=1$ terms from the first summation
on the right hand side of \cref{slmtranslation1}.
We first calculate the contribution from the first summation on the right hand side of
\cref{slmtranslation1}. Combining \cref{1formconstruction}, \cref{gfexpansion}, \cref{cdlp2form},
and \cref{slmtranslation1}, we obtain the expression for $\cdlpoform$ given by
\cref{cdlp1formrep}--\cref{qjkdef}. 
Applying integration by parts along $\partial P$, we obtain an alternative
expression for $\qjk$ in \cref{qjkdef}:
\be\label{qjkdef2}
\qjk(\brp,\br)
=\left(0,\frac{\brp-\br}{|\brp-\br|}\right)
\left(0,d(\nabla S^{(j,k)}(\br))\right).
\ee
And it is easy to see that the contribution from the $j=1$ terms on the right hand
side of \cref{slmtranslation1} is exactly zero due to the fact that $d(\nabla S^{1,k}(\br))=0$.
Thus, we have
\be\label{cdlp1form1}
d\cdlpoform(\brp,\br)=(0, \nabla G(\brp,\br))(0,\bnu(\br)) (0, \ntslm(\br))da(\br).
\ee
\paragraph{DLP 1-forms for constant density}
We now derive expressions for $\gammalm$ in \cref{cdlp1form}. It is clear that $\nslm(\brp)$
is a constant vector with respect to the integration variable $\br$. And
it is easy to verify that
\be
(0, \nabla G(\brp,\br))(0,\bnu(\br))da(\br)
= (-\nabla G(\brp,\br)\cdot \bnu(\br), \nabla G(\brp,\br)\times \bnu(\br))da(\br)
\ee
is closed and exact. Thus, we have \cref{gammalmdef}. Direct calculation shows that
\be
\frac{1}{4\pi} d\bomega = \nabla G(\brp,\br)\times \bnu(\br)da(\br).
\ee

Now, it is well-known that $-\int_P \nabla\frac{1}{|\brp-\br|}\cdot \bnu(\br)da(\br)$
is the solid angle of the patch $P$ with respect to the target point $\brp$ (see, for
example, \cite{asvestas1985josa1,asvestas1985josa2,masket1957rsi}). And \cref{omega0def}
can be found in \cite{asvestas1985josa1} with
\be
\frac{1}{4\pi} d\omega_0 = -\nabla G(\brp,\br)\cdot \bnu(\br)da(\br).
\ee
\end{proof}

\begin{remark}
  By \cref{hlmslmconnection}, it follows that
the real 1-forms $\dlpoform$ in \cref{dlp1form} are given by 
\be\label{dlp1formrep2}
\dlpoform(\brp,\br) = \sqrt{2} \Im\left(\cdlpoform(\brp,\br)+\gammalm(\brp,\br)\right).
\ee
\end{remark}

\Bl
\begin{remark}
  The condition $(\brp-\br)\cdot\bu\ge 0$ for any $\br\in P$ in \cref{thm:dlp1form}
  requires that $\bu$ points from the patch $P$ toward the target $\brp$.
  The solid angle $\Omega_0$ in \cref{solidangle} is discontinuous as the
  target $\brp$ crosses the patch $P$, and the condition on $\bu$ ensures that the line integral correctly
  evaluates the solid angle for the given target $\brp$. For on-surface evaluation, the line integral yields a
  one-sided limit rather than the Cauchy principal value; the principal
  value can then be obtained via the standard DLP jump relation.
\end{remark}
\Bk

Similarly, we consider the following complex 2-forms for the SLP rather than
the real 2-forms $\slptform$ in \cref{slp2form}:
\be\label{cslp2form}
\cslptform(\brp,\br) 
=(G(\brp,\br)\nslm(\br) - \nabla G(\brp,\br) \slm(\br))\cdot \bnu da(\br).
\ee
\begin{theorem}[SLP 1-forms]\label{thm:cslp1form}
  Let $\cslptform$ be complex differential 2-forms defined by \cref{cslp2form}.
  Suppose that the complex differential 1-forms
  $\cslpoform$ are defined by
\be\label{cslp1formrep}
\cslpoform(\brp,\br)=\frac{1}{4\pi} 
\sum_{j=1}^l D_{j,l} \sum_{k=-j}^j
\vjk(\brp,\br)
S^{(l-j,m-k)}(\brp)T_{jk,lm},
\ee
where
\be
D_{j,l}=\int_0^1 t^{j-1}(1-t)^{l-j}dt=\frac{(j-1)!(l-j)!}{l!},
\ee
and
\be\label{vjkdef}
\vjk(\brp,\br)=
\frac{1}{|\brp-\br|}
\left[(\brp-\br)\times\nabla S^{(j,k)}(\br)\right]\cdot d\br.
\ee  
Then
\be\label{cslp1form}
d\left(\cslpoform(\brp,\br)+\slm(\brp)\omega_0(\brp,\br)\right) = \cslptform(\brp,\br),
\ee
where $\omega_0$ is given by \cref{omega0def}.
\end{theorem}
\begin{proof}
  In order to apply \cref{1formconstruction}, we split $\cslptform$ into two terms:
\be\label{cslp2form1}
\ba
\cslptform(\brp,\br) 
&=(G(\brp,\br)\nslm(\br) - \nabla G(\brp,\br) \stlm(\br))\cdot \bnu da(\br)\\
&- \slm(\brp)\nabla G(\brp,\br)\cdot \bnu da(\br),
\ea
\ee
where
\be
\stlm(\br) = \slm(\br)-\slm(\brp).
\ee
Consider first 1-forms for the first term on the right hand side of \cref{cslp2form1}.
It follows from \cref{slmtranslation0} that
\be\label{slmtranslation2}
\ba
\stlm(\brp+t(\br-\brp))
&=\sum_{j=1}^l t^j(1-t)^{l-j} \sum_{k=-j}^j
S^{(j,k)}(\br)S^{(l-j,m-k)}(\brp)T_{jk,lm}\\
&+((1-t)^l-1)
S^{(0,0)}S^{(l,m)}(\brp)T_{00,lm},\\
\nslm(\brp+t(\br-\brp)) 
&=\sum_{j=1}^l t^{j-1}(1-t)^{l-j} \sum_{k=-j}^j
\nabla S^{(j,k)}(\br)S^{(l-j,m-k)}(\brp)T_{jk,lm}.
\ea
\ee

Combining \cref{1formconstruction}, \cref{gfexpansion}, \cref{cslp2form1}, 
and \cref{slmtranslation2}, we obtain the expression for $\cslpoform$ given by
\cref{cslp1formrep}--\cref{vjkdef}.
For the second term on the right hand side of \cref{cslp2form1}, it is easy to see that
\be
-\slm(\brp)\nabla G(\brp,\br)\cdot \bnu da(\br) = \slm(\brp)d \omega_0,
\ee
with $\omega_0$ defined in \cref{omega0def}.
\end{proof}

\begin{remark}
  By \cref{hlmslmconnection}, it follows that
  the real 1-forms $\slpoform$ in \cref{slp1form} are given by
\be
\slpoform(\brp,\br) =
\sqrt{2} \Im\left(\cslpoform(\brp,\br) + \slm(\brp) \omega_0(\brp,\br)\right).
\ee
\end{remark}

\subsection{1-to-0 form conversion}\label{sec:1formeval}
We now present some details on the evaluation of line integrals
(i.e., integral of 1-forms)
via the singularity-swap quadrature in \cite{afklinteberg2021bit}. Pulling out
the target-dependent part from the 1-forms in \cref{cdlp1formrep}--\cref{bomegadef}
for the DLP and \cref{cslp1formrep}--\cref{vjkdef} for the SLP, we observe that
the following line integrals need to be calculated:
\be
\ba
\Qjk(\brp)&=\int_{\partial P} \qjk(\brp,\br)
=-\int_{\partial P} \left(0,d\left(\frac{\brp-\br}{|\brp-\br|}\right)\right)
\left(0,\nabla S^{(j,k)}(\br)\right),\\
\bOmega(\brp)&=\int_{\partial P} \bomega(\brp,\br)
=-\int_{\partial P} \frac{d\br}{|\brp-\br|},\\
\Vjk(\brp)&=\int_{\partial P} \vjk(\brp,\br)
=\int_{\partial P}\frac{1}{|\brp-\br|}
\left[(\brp-\br)\times\nabla S^{(j,k)}(\br)\right]\cdot d\br,
\ea\label{lineintegrals}
\ee
and
\be\label{solidangle}
\Omega_0=\int_{\partial P} \omega_0(\brp,\br)
=\int_{\partial P}\frac{((\brp-\br)\times\bu)\cdot d\br}{|\brp-\br|(|\brp-\br|+\bu\cdot(\brp-\br))}.
\ee
We will use $\Vjk$ to illustrate the steps in the evaluation of line integrals in \cref{lineintegrals}. For a smooth triangular patch $P$, its boundary $\partial P$ consists of three smooth curves. Denote one such boundary curve by $\Gamma$, and assume that the parameterization of $\Gamma$ induced by the parameterization of $P$ is $\br(t)$ with $t\in[-1,1]$. Then we only need to consider the evaluation of the following integral:
\be
I^{(j,k)}
(\brp) = \int_{-1}^1 \frac{1}{|\brp-\br(t)|}
\left[(\brp-\br(t))\times\nabla S^{(j,k)}(\br(t))\right]\cdot\frac{d\br(t)}{dt} dt.
\ee
As discussed in \cref{sec:ssq}, the above integral can be written as
\be
I^{(j,k)}(\brp) = \int_{-1}^1 \frac{F(t)}{((t-t_0)(t-\overline{t_0}))^{1/2}}dt,
\ee
where $t_0$ is a complex root satifying \cref{complexparameter} and
\be
F(t) = \left[(\brp-\br(t))\times\nabla S^{(j,k)}(\br(t))\right]\cdot\frac{d\br(t)}{dt}
\frac{|t-t_0|}{|\br(t)-\brp|}.
\ee
For a fixed target point $\brp$, we approximate $F(t)$ by its polynomial interpolant
with, say, the $\tp$ Gauss-Legendre nodes $t_i$ ($i=1,\ldots,\tp$)
as the interpolation nodes. That is,
\be
F(t)=\sum_{i=0}^{\tp-1}\tc_i t^i, \quad \tc = U F,
\ee
where
\be
\tc=\begin{bmatrix}\tc_0\\ \vdots \\ \tc_{\tp-1}\end{bmatrix},
\quad
F=\begin{bmatrix}F(t_1)\\ \vdots \\ F(t_\tp)\end{bmatrix},
\quad
U=\begin{bmatrix}1& t_1& \cdots & t_1^{\tp-1}\\
\vdots &\vdots&\vdots&\vdots\\
1& t_{\tp}& \cdots & t_{\tp}^{\tp-1}\end{bmatrix}^{-1}.
\ee
Here, $\tp$ should be slightly greater than $p$ in order to take the geometric
variation and the smooth factor $|t-t_0|/|\br(t)-\brp|$ into account.
In practice, we observe that $\tp=p+5$ or $\tp=2p$ works well.
The integrals
\be\label{ssq1dintegrals}
\mathcal{I}_i(\brp) = \int_{-1}^1 \frac{t^i}{((t-t_0)(t-\overline{t_0}))^{1/2}}dt
\ee
can be calculated by a recurrence relation (see \cite{afklinteberg2021bit} for details).
Thus,
\be\label{ssq1dintegrals2}
I^{(j,k)}(\brp) = \sum_{i=0}^{\tp-1}\tc_i \mathcal{I}_i
=\mathcal{I}(\brp)^TUF(\brp), \quad
\mathcal{I}(\brp)^T=\begin{bmatrix}
\mathcal{I}_0(\brp) & \ldots & \mathcal{I}_{\tp-1}(\brp)\end{bmatrix}.
\ee
Note that the matrix $U$ can be precomputed and stored since it is independent of $\brp$.

After quantities such as $I^{(j,k)}$ have been calculated, the line integrals
associated with the DLP and SLP can be calculated easily. For example, the integrals
of $\cdlpoform$ in \cref{cdlp1formrep} are computed by
\be\label{cdlp1dintegral}
I[\cdlpoform](\brp)=\int_{\partial P}\cdlpoform(\brp,\br)=\frac{1}{4\pi} 
\sum_{j=2}^l C_{j,l} \sum_{k=-j}^j
\Qjk(\brp)
S^{(l-j,m-k)}(\brp)T_{jk,lm}.
\ee

\begin{remark}
  Straightforward calculation shows that \cref{qjkdef} contains
  both $1/|\brp-\br|$ and $1/|\brp-\br|^3$ near-singularities. If we use
  \cref{qjkdef2} to calculate $\Qjk$, then we only need to deal with
  the $1/|\brp-\br|$ near-singularity.
\end{remark}

\begin{remark}
  The solid angle $\Omega_0$ in \cref{solidangle}
  does not have a clean $1/|\brp-\br|^m$ near-singularity. However, if we pull out
  the factor $|\brp-\br|^2$ from the denominator in \cref{solidangle}, we obtain
  $(1+\cos\varphi)$ with $\varphi$ the angle between $\bu$ and $\brp-\br$. By the condition stated in \cref{thm:dlp1form} on $\bu$, $\cos\varphi$ is greater than $0$ in most cases. Thus, the extra factor $(1+\cos\varphi)$ can be treated by a slight upsampling or very mild adaptive integration.
\end{remark}

\subsection{Evaluation of the normal derivative of the single layer potential}
\label{sec:sprimeeval}
We denote the quaternions $(0,\bnu')$ and $(0,\bnu)$ by $\nu'$ and $\nu$, respectively.
Since $\bnu$ is a unit vector, we have $\nu \bar{\nu}=1$.
In order to evaluate
$\sprime[\sigma](\brp)=-\int_P\bnu'\cdot\nabla G(\brp,\br) \sigma(\br)da(\br)$,
where $\nabla$ is with respect to $\br$,
we rewrite the integrand using quaternion algebra as follows:
\be
\ba
\bnu'\cdot\nabla G(\brp,\br) \sigma(\br)da(\br)
&=\left[(0,\bnu') (0,\nabla G(\brp,\br)) \sigma(\br)da(\br)\right]_0\\
&=\left[\nu'(0,\nabla G(\brp,\br))\nu (\bar{\nu} \sigma(\br))da(\br)\right]_0.
\ea\label{sprime2formrep}
\ee
We observe that $\nu'$ depends on the target point $\brp$ and thus can be pulled
out the integration. We now treat $\bar{\nu}\sigma=(0,-\bnu(\br))\sigma(\br)$
as a quaternion density, and approximate it via harmonic polynomials:
\be\label{sprimedensappr}
(0,\sigma(\br)\nu(\br)) \approx \sum_{l=1}^p\sum_{m=1}^l
(0, \nabla H^{(l,m)}(\br)) \wlm, \quad \br \in P.
\ee
And the linear system for obtaining the quaternion coefficient $\clm$
is given by
\be
\label{sprimedensitylinearsystem}
\sum_{l=1}^p\sum_{m=1}^l
(0, \nabla H^{(l,m)}(\br^{(i,j)})) \wlm = (0,\sigma(\br^{(i,j)})\bnu(\br^{(i,j)})), \quad 0\le i+j \le p-1.
\ee
We observe that \cref{sprimedensappr} is very similar to the quaternion harmonic density
approximation \cref{dlpdensappr} for the double layer potential, and that \cref{sprimedensitylinearsystem}
has the same system matrix as \cref{dlpdensitylinearsystem} for the DLP evaluation.
Thus, the evaluation of $\sprime$ is very close that of $\dlp$. The difference is that
the right hand side in the density approximation is $(\mu,\bzero)$ in \cref{dlpdensappr}
for $\dlp$, and $(0,\sigma\bnu)$ in \cref{sprimedensappr} for $\sprime$. We also need
to multiply the quaternion $\nu'=(0,\bnu')$ from left for $\sprime$ in postprocessing.
In short, we have
\be
\ba
\sprime[\sigma](\brp) 
&\approx \left[(0,\bnu')\sum_{l,m} \int_p\dlptform(\brp,\br) \wlm\right]_0\\
& = \left[(0,\bnu')\sum_{l,m} \int_{\partial P}\dlpoform(\brp,\br) \wlm\right]_0,
\ea\label{sprime1formrep}
\ee
where $[\cdot]_0$ denotes the scalar part of the quaternion,
2-forms $\dlptform$ are defined by \cref{dlp2form}, and 1-forms $\dlpoform$
are defined by \cref{dlp1form}, whose calculation is discussed in detail in
\cref{thm:dlp1form} and \cref{sec:1formeval}.

\subsection{Evaluation of the normal derivative of the double layer potential}
\label{sec:dprimeeval}
The normal derivative, or the gradient, of the double layer potential
can be computed by differentiating the associated 1-forms. By combining
\cref{dlp1formrep},
\cref{cdlp1formrep}--\cref{bomegadef} and \cref{dlp1formrep2}, we have
\be
\ba
\dlp[\mu](\brp)& \approx
-\sum_{l,m} \int_{\partial P} \left[\dlpoform(\brp,\br) \clm\right]_0\\
&=-\sqrt{2} \left[\sum_{l,m}\Im\left(\int_{\partial P}
  \left(\cdlpoform(\brp,\br)+\gammalm(\brp,\br)\right)\right)\clm\right]_0\\
&:=-\sqrt{2} \left[\sum_{l,m}\Im
  \left(I[\cdlpoform](\brp)+I[\gammalm](\brp)\right)\clm\right]_0.
\ea
\ee
Here, $I[\cdlpoform]$ is given by \cref{cdlp1dintegral} and $I[\gammalm]$
is defined by the formula
\be
I[\gammalm](\brp)=\int_{\partial P}\gammalm(\brp,\br)=\frac{1}{4\pi} 
\Omega(\brp) (0, \nabla \slm(\brp)),
\ee
where $\Omega(\brp)=(\Omega_0(\brp),\bOmega(\brp))$ with $\Omega_0$ and
$\bOmega$ given by \cref{solidangle} and the second equation in \cref{lineintegrals},
respectively. Product rule can then be applied to differentiate $\Omega(\brp)$
and $\nabla\slm(\brp)$ to obtain the derivatives of $I[\cdlpoform](\brp)$
and $I[\gammalm](\brp)$.
Straightforward calculation leads to the following expressions:
\be\label{dprimelineintegrals}
\ba
\frac{\partial \Qjk(\brp)}{\partial \bnu'}
&=\int_{\partial P} \left(0,\frac{\bnu'}{|\brp-\br|}
-\frac{\left(\bnu'\cdot(\brp-\br)\right)(\brp-\br)}{|\brp-\br|^3}\right)
\left(0,d\nabla S^{(j,k)}(\br)\right),\\
\frac{\partial \bOmega(\brp)}{\partial \bnu'}
&=\int_{\partial P} \frac{\bnu'\cdot(\brp-\br) d\br}{|\brp-\br|^3},\\
\frac{\partial \Omega_0(\brp)}{\partial \bnu'}
&=\int_{\partial P}\frac{(\bnu'\times\bu)\cdot d\br}{|\brp-\br|(|\brp-\br|+\bu\cdot(\brp-\br))}
-\frac{\left(\bnu'\cdot\bu\right) ((\brp-\br)\times\bu)\cdot d\br}{|\brp-\br|(|\brp-\br|+\bu\cdot(\brp-\br))^2}\\
&-\frac{\left(\bnu'\cdot(\brp-\br)\right)((\brp-\br)\times\bu)\cdot d\br}{|\brp-\br|^2(|\brp-\br|+\bu\cdot(\brp-\br))^2}
-\frac{\left(\bnu'\cdot(\brp-\br)\right)((\brp-\br)\times\bu)\cdot d\br}{|\brp-\br|^3(|\brp-\br|+\bu\cdot(\brp-\br))}.
\ea
\ee

We observe that both $\partial \Qjk(\brp)/\partial \bnu'$
and $\partial \bOmega(\brp)/\partial \bnu'$ can be computed directly
via the singularity-swap quadrature, while
$\partial \Omega_0(\brp)/\partial \bnu'$ can be computed
via the singularity-swap quadrature with slight upsampling or mild adaptive integration.

\section{Numerical algorithm}\label{sec:algorithm}
The following algorithm summarizes the steps in the evaluation of the Laplace
double layer potential.

\vspace{.2in}

\noindent
\begin{center}
{\bf Close evaluation of the Laplace double layer potential}
\end{center}
\vspace{.2in}

{{\bf Input/output}:} 
Suppose that $P$ is a smooth triangular patch on $S$ (we assume that translation and rotation have been carried out so that three vertices
of $P$ are on the $xy$ plane and \Bl the mean of the three vertices \Bk is the origin),
$\mu$ is the given DLP density, and $\brp$ is
the given target point. Let $p$ be the polynomial approximation order on $P$
and $\np=p(p+1)/2$,
and $\br^{(i,j)}$ (\Rd $0\le i+j\le p-1$ \Bk) be a set of interpolation/collocation nodes on $P$.
Let $\tp$ (say, $\tp=2p$) be the polynomial approximation order on $\partial P$, and
$\br_i$ ($i=1,\ldots,3\tp$) be a set of interpolation nodes on three edges
of $\partial P$.
Compute the DLP
$\dlp[\mu](\brp)=\int_P \nabla G(\brp,\br)\cdot\bnu(\br)\mu(\br) da(\br)$.

\vspace{.1in}

\noindent
{{\bf Step 1}:} For each source point $\br^{(i,j)}$ on $P$,
compute $\slm(\br^{(i,j)})$ and $\nslm(\br^{(i,j)})$
for all $1\le m\le l\le p$ via the recurrence relation \cref{associatelegendrerecurrence}. Construct
the system matrix in \cref{dlpdensitylinearsystem2}. 
[{\em This step depends only on sources.
    The cost of constructing the system matrix is $O(\np^2)=O(p^4)$.}]

\vspace{.1in}

\noindent
{{\bf Step 2}:} Solve the linear system \cref{dlpdensitylinearsystem2} to obtain
the quaternion coefficient vector $c^{(l,m)}$ for all $1\le m\le l\le p$.
[{\em This step depends only on sources.
    The cost of solving the $4\np\times 4\np$ linear system is $O((4\np)^3)=O(p^6)$.}]

\vspace{.1in}

\noindent
{{\bf Step 3}:} For each source point $\br_i$ on $\partial P$,
compute $\slm(\br_i)$ and $\nslm(\br_i)$ for all $1\le m\le l\le p$ via
the recurrence relation \cref{associatelegendrerecurrence}. 
[{\em This step depends only on sources.
    The cost is $O(p^3)$.}]

\vspace{.1in}

\noindent
{{\bf Step 4}:} For the target point $\brp$, compute $\slm(\brp)$ and $\nslm(\brp)$ via
the recurrence relation \cref{associatelegendrerecurrence}. 
[{\em This step depends only on the target.
    The cost is $O(p^2)$.}]

\vspace{.1in}

\noindent
{{\bf Step 5}:} For each edge on $\partial P$, find the value of $t_0$ for the
given target $\brp$ and compute $\mathcal{I}_i$ in \cref{ssq1dintegrals}
for $0\le i\le \tp-1$ via the recurrence relation in \cite{afklinteberg2021bit}.
[{\em This step depends on the target and sources on $\partial P$.
    The cost is $O(p)$.}]

\vspace{.1in}

\noindent
{{\bf Step 6}:} For the target point $\brp$, compute the line integrals
$\Qjk(\brp)$ in \cref{lineintegrals} for all $1\le k\le j\le p$ using formulas similar to \cref{ssq1dintegrals2}.
[{\em This step depends on the target and sources on $\partial P$.
    For each $(j,k)$ pair, the work \Bl involves \Bk elementwise multiplication and
    a matrix-vector product. Thus, the cost of this step is $O(p^3)$.}]

\vspace{.1in}

\noindent
{{\bf Step 7}:} For the target point $\brp$, compute the integrals
$I[\cdlpoform](\brp)$ in \cref{cdlp1dintegral} for all $1\le m\le l \le p$. Also compute
$I[\gammalm](\brp)$, the integrals of $\gammalm$, for $\gammalm$ defined in
\cref{gammalmdef}. 
[{\em This step depends on the target and sources on $\partial P$.
    Using the fast translation scheme in \cref{rem:fasttranslation}, the cost of
    this step is $O(p^3)$.}]

\vspace{.1in}

\noindent
{{\bf Step 8}:} Compute the DLP by combining \cref{dlp1formrep},
\cref{dlp1formrep2}, and \cref{cdlp1dintegral}. That is, compute
\be\label{wlmdef}
\Dlm(\brp) = -\sqrt{2}\Im\left(I[\cdlpoform](\brp)+I[\gammalm](\brp)\right),
\quad 1\le m \le l \le p,
\ee
then set
\[\dlp[\mu](\brp)=\sum_{l,m}\left[\Dlm(\brp)\clm\right]_0,\] 
where $[\cdot]_0$ denotes the scalar part of the quaternion.
[{\em This is basically the evaluation of the quaternion product for two quaternion
    vectors. The cost of this step is $O(p^2)$.}]

\begin{remark}
 The above algorithm works for self-interaction, near-interaction, and close
 evaluation outlined in \cref{sec:intro}, providing a unified high-order scheme
 for all three interactions.
\end{remark}

\begin{remark}
  The above algorithm assumes that the density is given. In the case of constructing
  the system matrix during the solve phase of integral equation methods,
  we first skip {\bf Step 2}. Then instead of computing the quaternion dot product of
  $w^T c$ in {\bf Step 8}, we compute $[w^T(\brp) A^{-1}]_0$ to obtain a row vector
  corresponding to the given target $\brp$. Here, $w^T=[\wlm], 1\le m\le l \le p$ is
  the quaternion
  row vector in \cref{wlmdef}, and $A$ is the system matrix constructed in {\bf Step 1}.
\end{remark}

\begin{remark}
(a)  The evaluation of the single layer potential is similar, with additional scalar
  density approximation in \cref{slpdensitylinearsystem} and calculation of
  integrals of 1-forms in \cref{cslp1formrep}.
(b) The evaluation of the normal derivative of the SLP is almost identical. In
  {\bf Step 2}, we solve the linear system \cref{sprimedensitylinearsystem} instead.
  In {\bf Step 8}, we evaluate the last expression in \cref{sprime1formrep}.
(c) The evaluation of the normal derivative of the DLP is almost identical.
  In {\bf Step 6}, we also need to compute line integrals in \cref{dprimelineintegrals}
  as well.
\end{remark}

\begin{remark}
  In order to evaluate layer potentials for a collection of target points, we
  will need to divide the surface $S$ into a collection of triangular patches and
  identify the targets that require special quadrature construction for each patch.
  These two
  steps are fairly standard in scientific computing and we omit the details. For example,
  the second step can be carried out using either an adaptive octree~\cite{biros2008sisc}
  or a $k$-d tree (see, for example, \cite{saye2014camcs,yu2015acm}).
\end{remark}

\begin{remark}
  The cost associated with each source patch is $O(p^6)$ with $O(p^2)$ source points
  on each patch. The cost associated with each target is $O(p^3)$. Thus, the total
  cost is $O(N_S p^4 + N_T p^3)$, with $N_S$ the total number of source points on $S$
  and $N_T$ the total number of target points requiring special quadrature. During the
  solve phase of integral equation methods, $N_T$ is a small factor of $N_S$
  (about $2-4$ for reasonable geometries). The actual performance depends on many other
  factors such as the quality of implementation, memory latency, and prefactors in
  lower order terms.
\end{remark}

\section{Numerical examples}\label{sec:numericalexamples}
In this section, we illustrate the performance of our new quadrature scheme. The code is
written in Fortran and compiled using gfortran 13.2.0 with optimization flag -O3.
All experiments were run in single-threaded mode on a 3.60GHz AMD EPYC 9474F CPU.

For typical applications in integral equation methods, the evaluation of layer potentials
is needed in two stages -- the solve phase and the evaluation phase. Below, we will use
the double layer potential as the example to explain the main steps, since other layer
potentials can be treated in an almost identical manner.
Denote the discrete matrix for the double layer potential by $\bD$.
During the solve phase, the double layer
potential operator is a map from the boundary to itself and $\bD$ is a square
matrix of
size $N_S\times N_S$, where $N_S$ is the total number of source points on the boundary.
In order to apply $\bD$ fast to a given vector, it is necessary to split it into two
parts:
\be\label{dlpsplitting}
\bD = \bD_{\rm smooth} + \bD_{\rm RRQ},
\ee
where $\bD_{\rm smooth}$ is obtained by discretizing $\dlp$ via smooth quadrature
on each patch, and $\bD_{\rm RRQ}$ is the quadrature correction part due to the singularity
or near-singularity of the kernel. To be more precise, the weights in $\bD_{\rm smooth}$ 
depend only on the source points (i.e., functions of row indices only), while the weights
in $\bD_{\rm RRQ}$ depend on both sources and targets (i.e., functions of row and column
indices).
Since the quadrature weights depend only on sources, they can be combined with
other source-dependent factors such as the jacobian into the given vector. This makes
fast algorithms such as the FMM directly applicable to accelerate the matrix-vector
product for $\bD_{\rm smooth}$ in $O(N_S)$ cost, with the prefactor depending on the prescribed
precision. On the other hand, $\bD_{\rm RRQ}$ is only nonzero for the targets that require
our special quadrature for a given source patch. For typical geometries, the number of
targets close to any patch is $O(1)$. Thus, $\bD_{\rm RRQ}$ is a sparse matrix and the cost
of the associated matrix-vector product is again $O(N_S)$, where the prefactor depends on
the quadrature order $p$, the quality of triangulation, and the geometry of the boundary.

During the evaluation phase, layer potentials are evaluated on a set of volumetric grids.
Thus, the discrete matrix $\bD$ is a rectangular matrix of size $N_T\times N_S$, where
$N_T$ is the total number of target points in the volume. For typical applications, the
mesh size $h$ for the discretization of the boundary surface and the volume is roughly
equal to each other. Thus, $N_T=O(1/h^3)$ and $N_S=O(1/h^2)$. The application of $\bD$ 
to a known density is almost identical to the solve phase, as outlined above, with the
cost of both parts being $O(N_S+N_T)$. The
difference is that there are now many more target points, leading to different ratios
in the computation time for these two parts. For both cases, we denote the apply
time of $\bD_{\rm smooth}$ using the FMM by $T_{\rm FMM}$ and the time for building
$\bD_{\rm RRQ}$
by $T_{\rm RRQ}$ (the apply time of $\bD_{\rm RRQ}$ is negligible).
Similar to \cite{greengard2021jcp},
We use the state-of-the-art software package \texttt{FMM3D}~\cite{fmm3d}
for the fast multipole method.
It is clear that $T_{\rm FMM}$ serves as an effective benchmark
for evaluating the efficiency of our quadrature scheme. We remark here that the total time
of the solve phase for iterative solvers such as GMRES~\cite{gmres} is approximately
$T_{\rm RRQ} + N_{\rm iter} T_{\rm FMM}$,
where $N_{\rm iter}$ is the total number of GMRES iterations required to reach a prescribed
tolerance.

The accuracy study is conducted in a similar manner. During the solve phase, the target
points coincide with the source points on the boundary, involving only self-interactions
and near interactions. In contrast, for the evaluation phase, the target points are
generally located off the boundary within the volume—this is referred to as close
evaluation.
As discussed in \cref{sec:intro},
close evaluation is particularly challenging due to the limited availability of analytic
techniques and because the target points can be arbitrarily close to the surface,
requiring multiple levels of refinement in adaptive integration to resolve the
near-singularity.

In the following examples, we consider
the exterior Dirichlet problem for the Laplace equation:
\be\label{exteriordirichlet}
\ba
\Delta u(\brp) &=0, \quad \brp \in D,\\
u(\brp) &= f(\brp), \quad \brp \in S,
\ea
\ee
where $S$ is a smooth closed surface, and $D$ is the exterior domain.
In order to check the accuracy and convergence order, 
the boundary data $f$ is chosen to be the restriction of an analytic solution
on $S$, and the analytic solution is
constructed as the superposition of the fundamental solutions with
point sources lying on the interior of $S$:
\be\label{exactsolution}
u_{\rm exact}(\brp) = \sum_{j=1}^{4} \frac{c_j}{|\brp - \br_j|}, \quad \br_j \in \bar{D}^c,
\ee
where $\br_j$ ($j=1,\ldots,4$) are their locations, and $c_j$ are their strengths.
Both $\br_j$ and $c_j$ are chosen randomly.

We use the so-called combined field integral
representation~\cite[Chapter 3]{ColtonKress98}
\be\label{integralrepresentation}
u(\brp) = (\slp+\dlp)[\sigma](\brp), \quad \brp\in D,
\ee
and the resulting boundary integral equation (BIE)
is
\be\label{extdiribie}
\frac{1}{2}\sigma(\brp)+(\slp+\dlp)[\sigma](\brp) = f(\brp), \quad \brp\in S.
\ee
The boundary is divided into triangle patches, where the triangulation is carried
out more or less uniformly in the parameter space. 
The Vioreanu-Rokhlin~\cite{vioreanu2014sisc} nodes as the collocation points on each
patch. The associated Vioreanu-Rokhlin quadrature is used to discretize the smooth
interaction, and our quadrature scheme is used to discretize self- and near- interactions.
The resulting linear system is solved via GMRES with prescribed tolerance $10^{-14}$.
A Cartesian tensor grid of equispaced points of certain size is
placed on a rectangular box enclosing $S$ and evaluations of the numerical solution
to \cref{exteriordirichlet}
are carried out at the grid points that are in the exterior
domain. The relative $l_\infty$ error is computed via the formula
\be\label{relativeerror}
E_\infty = \frac{\| u_{\rm num}-u_{\rm exact}\|_\infty}{\|u_{\rm exact}\|_\infty}.
\ee
Note that $E_\infty$ is the cumulative error from both the solve phase and the evaluation
phase.

\Rd
\subsection{Study of the quaternion approximation}
The first step of RRQ is to find the quaternion harmonic approximation for the given density by solving either \eqref{dlpdensitylinearsystem2} or \eqref{slpdensitylinearsystem}, where the system matrix $A$ is given explicitly in \eqref{dlpdensitylinearsystem2} in $4\times 4$ block form. Since the harmonic polynomials are defined in the ambient space $\mathbb{R}^3$ instead of the parameter space as in the usual polynomial approximation, it is useful to determine how the conditioning of the system matrix depends on the curvature of the patch and the convergence order $p$, and whether the condition number of the system matrix imposes an upper bound on $p$. 

As stated in Step 1 of \cref{sec:algorithm}, we scale, translate, and rotate each patch so that three vertices of the transformed patch lie on the $xy$-plane, with the mean of the three vertices at the origin and the longest side length equal to $1$. The scaling step converts similar triangular patches on spheres of different radii to exactly the same patch, thereby largely removing the effect of curvature on the condition number of $A$. Thus, to study the effect of curvature on the conditioning of the system matrix $A$, we consider the following paraboloid surface:
\be
\label{eq:paraboloid}
\boldsymbol{r}(u,v) = 
\begin{bmatrix}
  u \\
  v \\
  \frac{H}{2}(u^2+v^2)
\end{bmatrix},\qquad (u,v)\in T_{h},
\ee
where $T_{h}$ is an equilateral triangle of side length $h$ centered at the origin. Straightforward calculation shows that the maximum mean curvature of the paraboloid surface \eqref{eq:paraboloid} is achieved at the origin with a value equal to $H$. 

\Cref{fig:conditionnumber1} shows the dependence of the condition number of the system matrix $A$ on the maximum mean curvature and the size of the patch. The left figure shows that the condition number $\kappa(A)$ increases as the maximum mean curvature increases and as the order $p$ increases. However, if we reduce the size of the patch, the patch is well approximated by its tangent plane (i.e., a flat patch). The right figure shows that $\kappa(A)$ decreases as the side length of the patch decreases. Indeed, reducing the patch size by a factor of $10$ leads to a condition number close to that of a flat patch.

\begin{figure}[!ht]
\centering
\includegraphics[height=44mm]{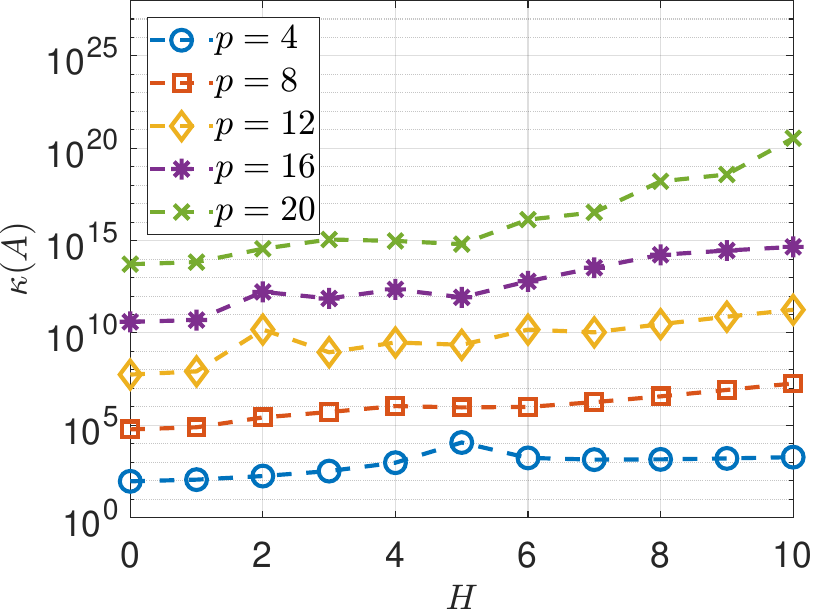} 
\hspace*{3mm}
\includegraphics[height=44mm]{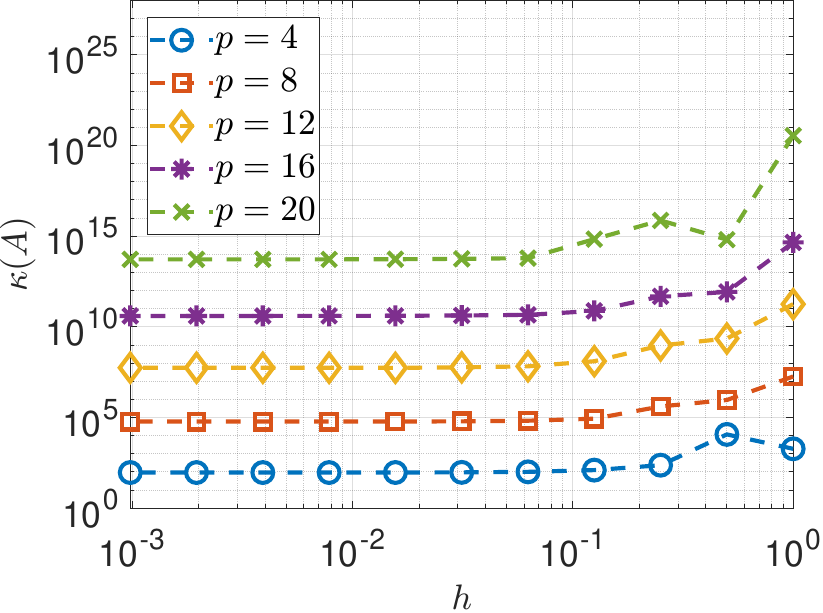}
\caption{\sf \Rd Study of the effect of curvature and patch size on the condition number of the system matrix $A$ in \eqref{dlpdensitylinearsystem2} and \eqref{slpdensitylinearsystem}. The triangular patch is taken from the paraboloid surface in \eqref{eq:paraboloid} with its projection to the $xy$-plane being an equilateral triangle of side length $h$ centered at the origin. The maximum mean curvature is achieved at the origin and is equal to $H$. Left: $h$ is fixed at 1, and $H$ is varied from 0 to 10. Right: $H$ is fixed at 10, and $h$ takes the value $2^{-L}$ for $L=0,\dots,10$. \Bk}
\label{fig:conditionnumber1}
\end{figure}

Nevertheless, the condition number increases rapidly as the order $p$ increases. This leads to the next natural question: whether the high condition number affects the accuracy of layer potentials, and if so, what is the upper limit of $p$ for a given error tolerance $\varepsilon$. In \cite{shen2025sinum}, the authors studied the effect of the condition number of the Vandermonde matrix $\kappa(V)$ on the accuracy of interpolation using a monomial basis in one dimension. They showed that $\kappa(V)$ has almost no effect on interpolation accuracy if a backward stable algorithm is used to solve the resulting linear system, provided that $\kappa(V) \le 1/\varepsilon_{\rm mach}$, where $\varepsilon_{\rm mach}$ is the machine epsilon (for IEEE double precision, $\varepsilon_{\rm mach} \approx 2.22 \times 10^{-16}$). 

A close examination of the analysis in \cite{shen2025sinum} indicates that many derivations are applicable to the quaternion harmonic approximation used here, and similar conclusions can be drawn for the interpolation error. Let $\mu$ be the given density on the right-hand side of \eqref{dlpdensitylinearsystem2}, $c$ be the computed quaternion coefficient vector (the numerical solution to $Ac = \mu$), and $q$ be the corresponding numerical quaternion harmonic polynomial approximation. Let $c_{\rm exact}$ be the exact solution to \eqref{dlpdensitylinearsystem2} and $q_{\rm exact}$ the corresponding exact quaternion harmonic polynomial approximation, both in the absence of floating-point error. A parallel derivation indicates that if $\kappa(A) \le 1/\varepsilon_{\rm mach}$, then the numerical interpolation error is bounded by
\be
\|\mu - q\|_\infty \le \|\mu - q_{\rm exact}\|_\infty + C \varepsilon_{\rm mach} \| c\|_2,
\label{eq:interpolationerror}
\ee
where $C$ is a positive constant depending on the Lebesgue constant of the interpolation operator. 

The first term on the right side of \eqref{eq:interpolationerror} depends on the density $\mu$ and the order $p$, but is independent of $\kappa(A)$. While the second term depends on $\kappa(A)$, it also depends on $\mu$; if $\mu$ does not excite higher modes, then $\| c\|_2$ may remain small even when $\kappa(A)$ is very large. We have carried out numerical experiments on a flat triangle (justified by our study on curvature effects); the results are summarized in \cref{fig:interperror}. We observe that $\kappa(A)$ reaches $1/\varepsilon_{\rm mach}$ at $p=26$. For the three functions studied in \cref{fig:interperror}, $\varepsilon_{\rm mach} \| c\|_2$ provides an excellent estimate of the highest achievable accuracy, and the stagnation point seems to be determined by this term. Indeed, for $\mu(x,y) = 1/(x^2+y^2+0.5)$, the interpolation error continues to decrease well beyond $p=26$. This subtle issue requires more detailed theoretical and numerical investigation, the findings of which will be reported at a later date.

\begin{figure}[!ht]
\centering
\includegraphics[height=44mm]{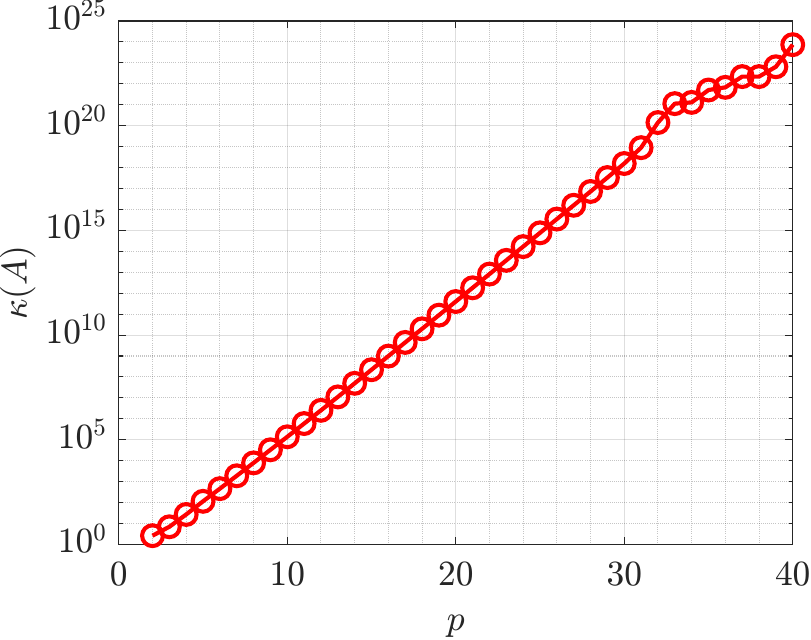} 
\hspace*{3mm}
\includegraphics[height=44mm]{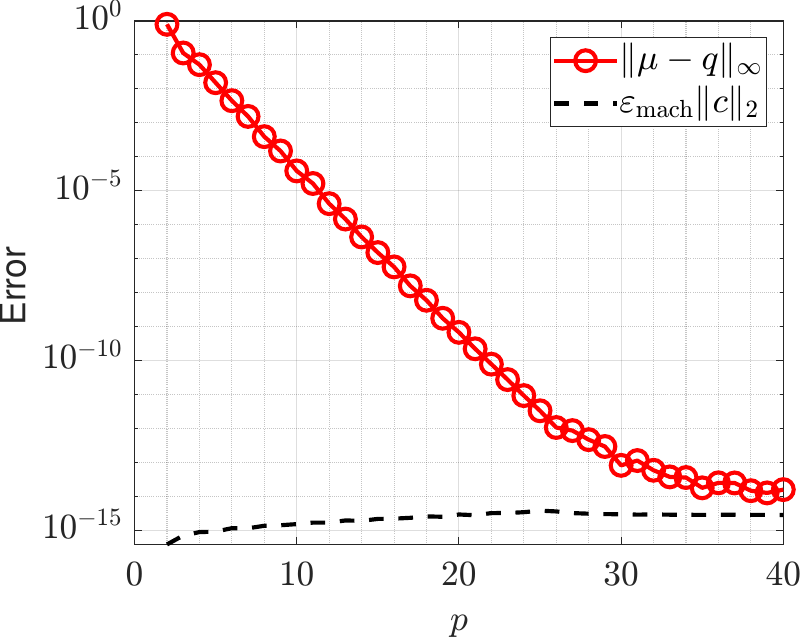}

\vspace*{2mm}
\includegraphics[height=44mm]{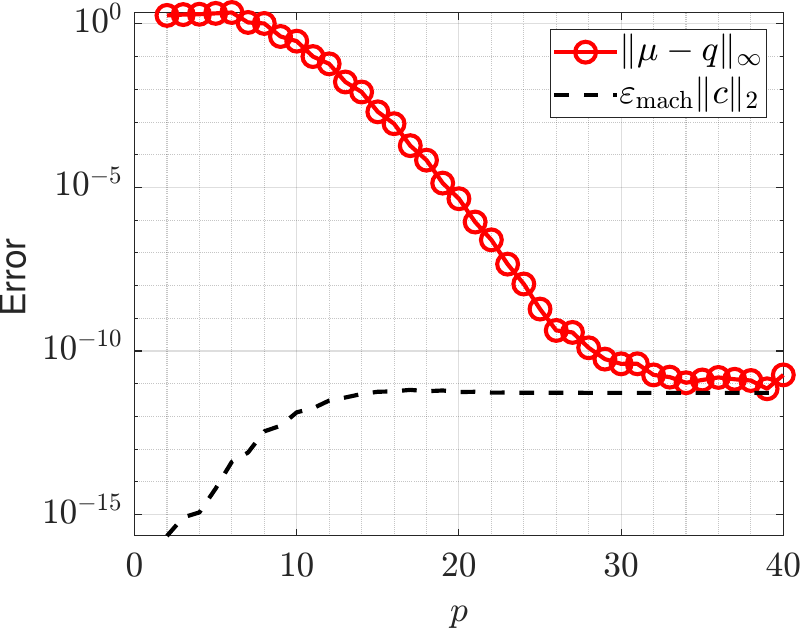}
\hspace*{3mm}
\includegraphics[height=44mm]{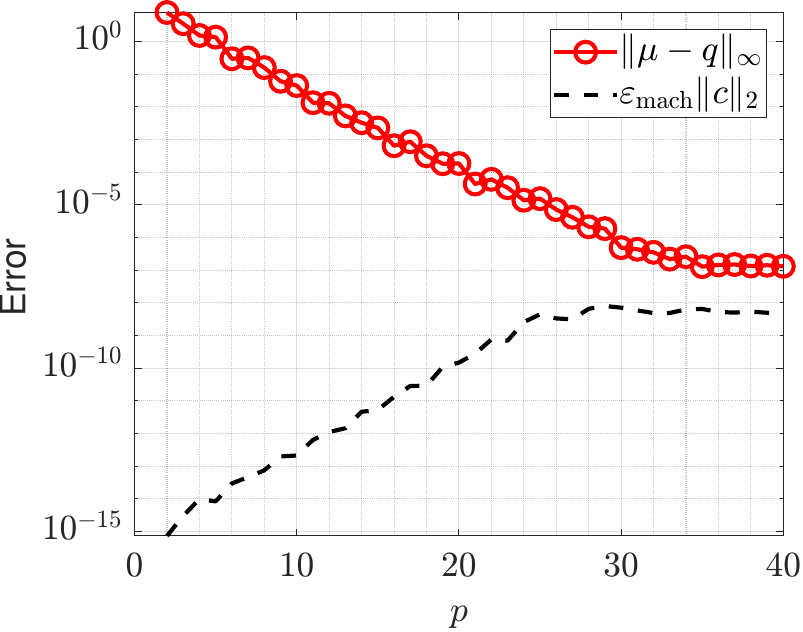}

\caption{\sf \Rd Study of quaternion interpolation error and the condition number of the system matrix $A$ in \eqref{dlpdensitylinearsystem2} versus order $p$ on a flat equilateral triangle. Top Left: $\kappa(A)$ as a function of $p$. Top Right: estimated interpolation error $\|\mu - q\|_\infty$ and $\varepsilon_{\rm mach} \| c\|_2$ for $\mu(x,y) = 1/(x^2+y^2+0.5)$. Bottom left: results for $\mu(x,y) = \cos(12x+12y+1)$. Bottom right: results for $\mu(x,y) = 1/(x^2+y^2+0.1)$. The interpolation error is estimated using approximately $20,000$ uniformly distributed points.\Bk}
\label{fig:interperror}
\end{figure}

\Bk

\subsection{Study of accuracy and convergence order}

\Rd
\subsubsection{Example 1: A unit sphere}\label{sec:sphere}
The boundary $S$ is a unit sphere
\be\label{eq:sphere}
S=\{\br\in\mathbb{R}^3:\ x^2+y^2+z^2=1\}.
\ee
We use the ``cubed sphere'' in \cite{ronchi1996jcp} to divide $S$ into
six charts, with each chart further divided into $2n^2$ triangles,
so that the total number of patches is $N_{\rm patches} = 12n^2$,
with $n\in\{2,4,\ldots,14\}$.
To isolate the quadrature error from the BIE solve error, we test the
Green's representation formula directly: for a harmonic
function $u$ in the interior of $S$, the identity
\be\label{eq:extinction}
\slp\!\left[\frac{\partial u}{\partial \bnu}\right](\brp) + \dlp[-u](\brp) = 0,
\quad \brp\in D,
\ee
holds for any point $\brp$ in the exterior domain $D$.
The density is taken to be a linear combination of solid harmonics of degree $8$,
\be\label{eq:sphereden}
u(x,y,z) = 0.70\,\Re(x+iy)^8 + z\,\Im(x+iy)^7,
\ee
and the relative $l_\infty$ error is
\be\label{eq:grferror}
E_\infty = \frac{\left\|\slp\!\left[\frac{\partial u}{\partial \bnu}\right] + \dlp[-u]\right\|_\infty}{\|u\|_\infty},
\ee
where the norms are evaluated on a uniform $102\times 102\times 102$ Cartesian grid
on $[-1.5,1.5]^3$ restricted to the exterior of $S$.

\Cref{table:sphere} shows $E_\infty$ at sample values of $n$ and $p$.
The image on the right side of \cref{fig:sphere} plots these data points as
a function of $N_{\rm patches}$ for $p=4,6,\ldots,14$, with dashed lines showing
the expected convergence order. Since the geometry is exactly smooth and the
test bypasses the BIE solve, the convergence curves are clean
and show clear separation between different orders, confirming that higher
values of $p$ do yield measurably higher accuracy.

\begin{figure}[!ht]
\centering
\includegraphics[height=44mm]{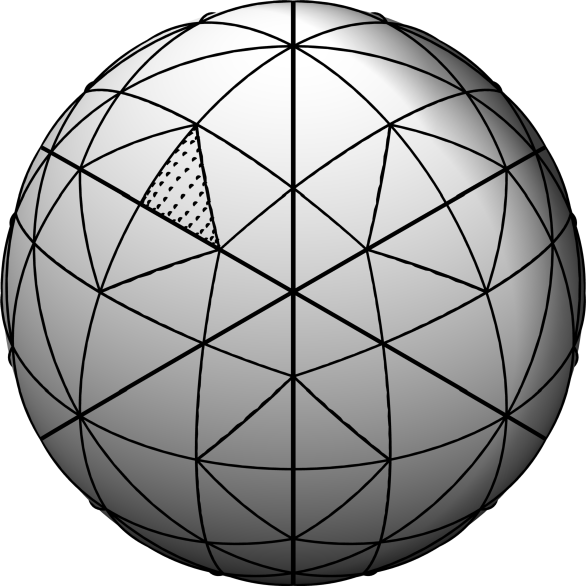}
\hspace*{2mm}
\includegraphics[height=44mm]{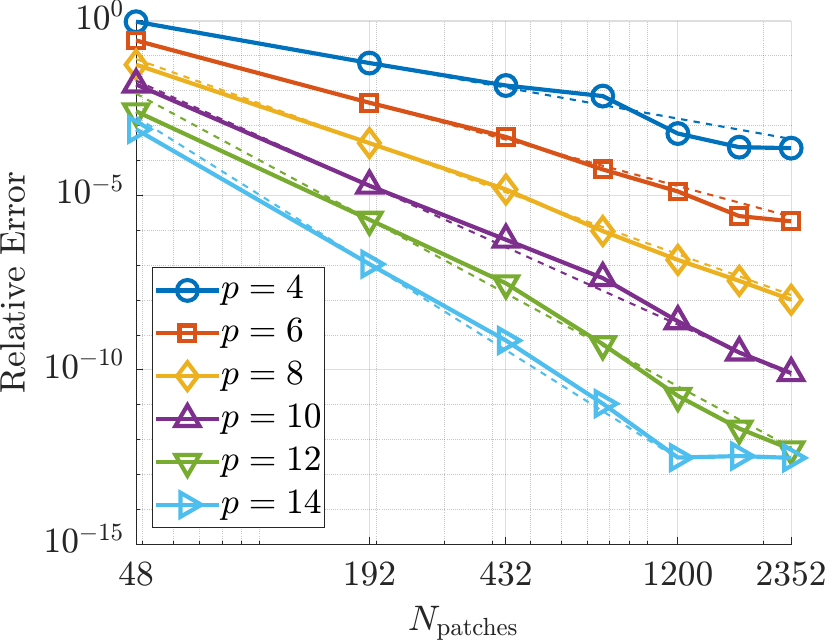}
\caption{\sf \Rd Left: triangulated unit sphere boundary. Right: relative $l_\infty$ errors of the Green's
  representation formula \eqref{eq:extinction} as functions of $N_{\rm patches}$
  for various orders $p$. Dashed lines indicate the $O(h^p)$ convergence rate,
  where $h = 1/n$ is the characteristic mesh size.\Bk}
\label{fig:sphere}
\end{figure}

\begin{table}[!ht]
\caption{\sf \Rd Relative $l_\infty$ errors $E_\infty$ for the unit sphere example at
  sample values of $n$ and $p$. The total number of patches is
  $N_{\rm patches}=12n^2$ and the total number of discretization points
  is $N=N_{\rm patches}\, p(p+1)/2$.\Bk}
\centering
\begin{tabular}{|c|c|c|c|c|c|c|}
\hline
{\diagbox{$12n^2$}{$p$}} & 4 & 6 & 8 & 10 & 12 & 14\\
\hline
$12\times 2^2$   & 9.54e-01 & 2.76e-01 & 5.62e-02 & 1.52e-02 & 2.59e-03 & 7.85e-04\\
$12\times 4^2$   & 6.24e-02 & 4.56e-03 & 3.17e-04 & 1.91e-05 & 2.03e-06 & 1.08e-07\\
$12\times 6^2$   & 1.42e-02 & 4.76e-04 & 1.49e-05 & 5.41e-07 & 3.05e-08 & 6.99e-10\\
$12\times 8^2$   & 7.06e-03 & 5.59e-05 & 9.54e-07 & 4.29e-08 & 5.49e-10 & 1.09e-11\\
$12\times 10^2$  & 5.96e-04 & 1.31e-05 & 1.42e-07 & 2.47e-09 & 1.81e-11 & 3.09e-13\\
$12\times 12^2$  & 2.42e-04 & 2.57e-06 & 3.58e-08 & 3.16e-10 & 2.10e-12 & 3.41e-13\\
$12\times 14^2$  & 2.27e-04 & 1.83e-06 & 1.03e-08 & 8.10e-11 & 5.31e-13 & 3.03e-13\\
\hline
\end{tabular}\label{table:sphere}
\end{table}
\Bk

\subsubsection{Example 2: A warped torus}
The boundary $S$ is a warped torus 
\be\label{eq:torus}
\br(\theta,\phi) = \begin{bmatrix}
\left(a+f(\theta,\phi)\cos(\theta)\right)\cos(\phi)\\
\left(a+f(\theta,\phi)\cos(\theta)\right)\sin(\phi)\\
f(\theta,\phi)\sin(\theta)
\end{bmatrix}, \quad (\theta,\phi)\in [0,2\pi]^2,
\ee
where $f(\theta,\phi) = b + \omega_c\cos(\omega_n\phi + \omega_m\theta)$. We set
the parameters to $a=1$, $b=1/2$, $\omega_c=0.065$, $\omega_n=5$ and $\omega_m=3$.
See the image on the left side of \cref{fig:torus}.

\begin{figure}[t]
\centering
\includegraphics[height=45mm]{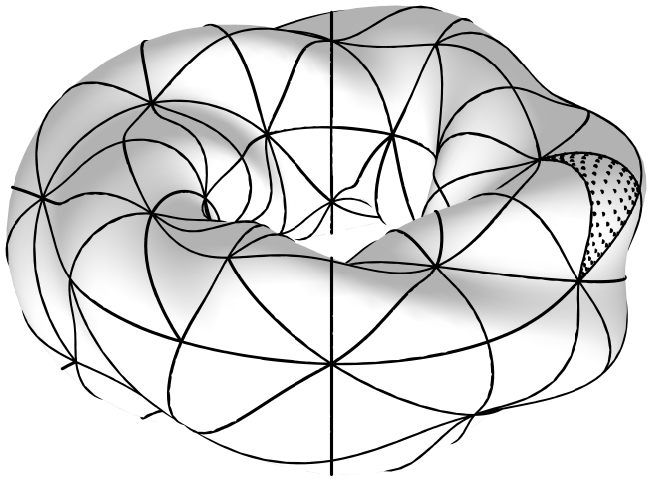} 
\hspace*{2mm}
\includegraphics[height=45mm]{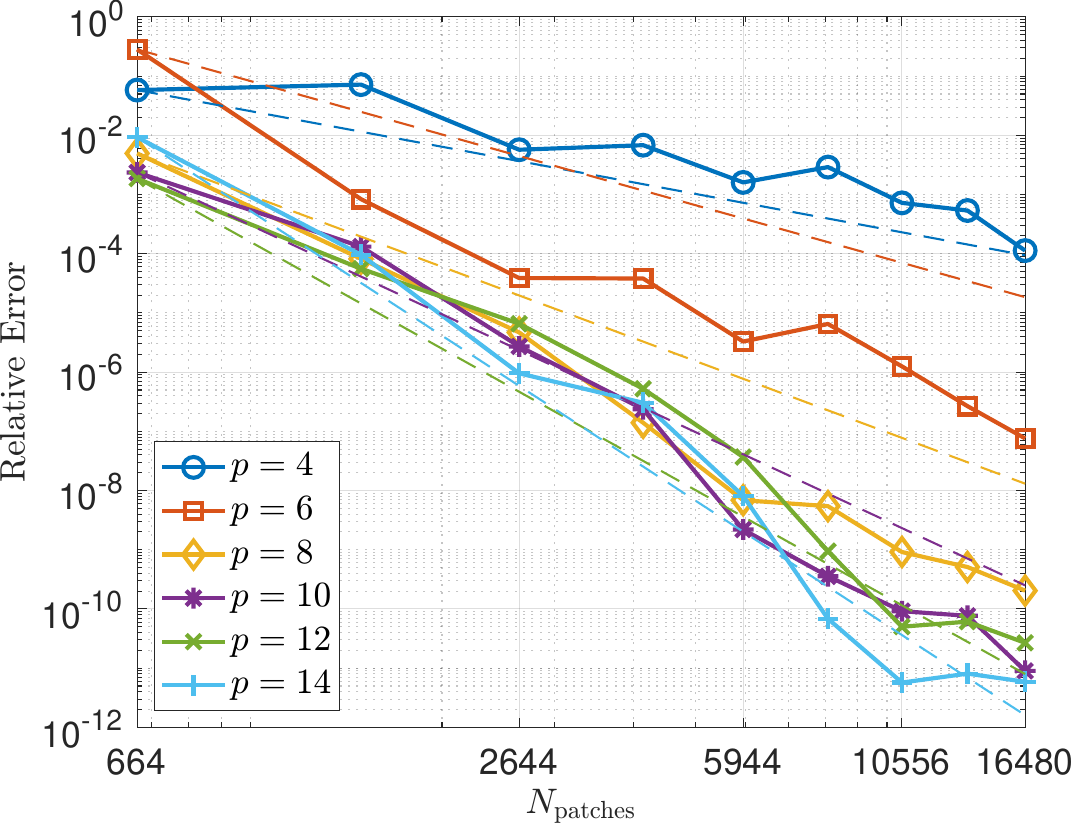}
\caption{\sf Left: triangulated warped torus boundary. Right:
  relative $l_\infty$ errors as functions of $N_{\rm patches}$ for various orders $p$.}
\label{fig:torus}
\end{figure}

\begin{figure}[!ht]
\centering
\includegraphics[height=45mm]{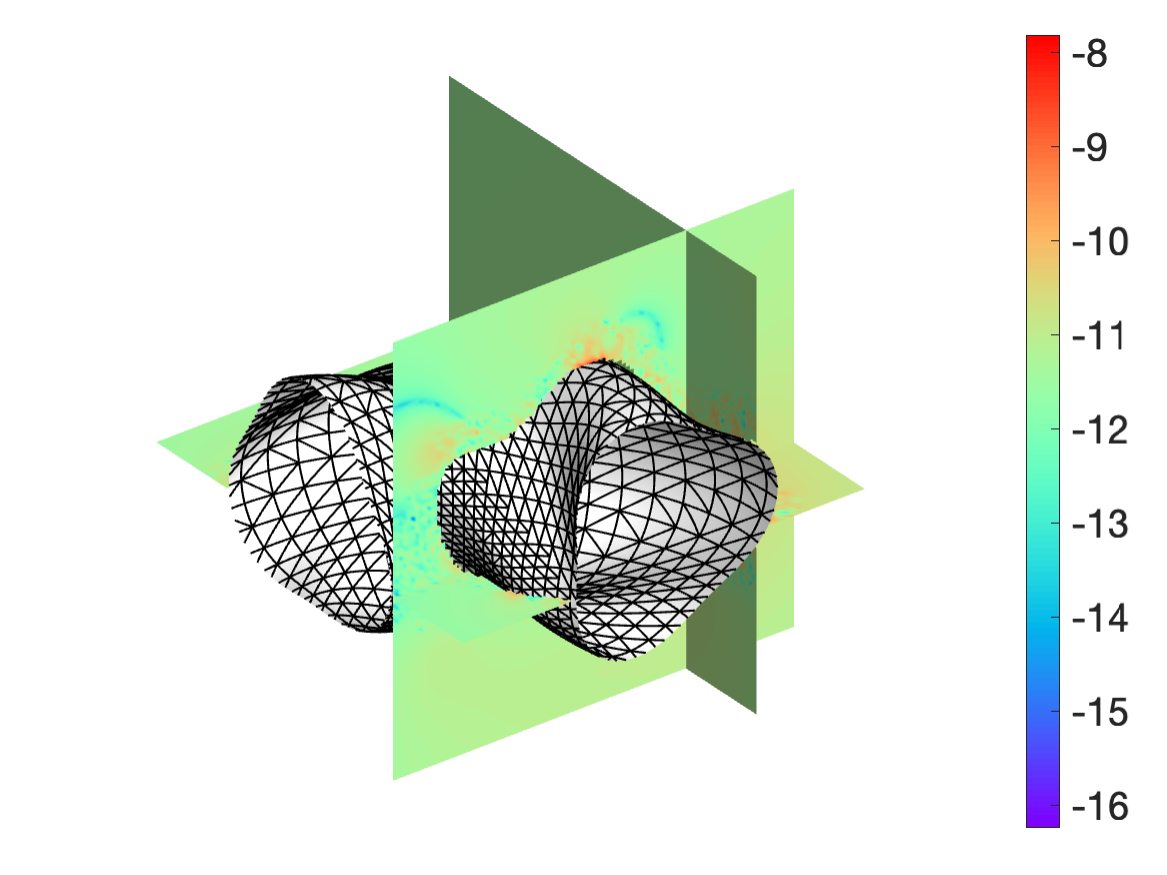} 
\hspace*{2mm}
\includegraphics[height=45mm]{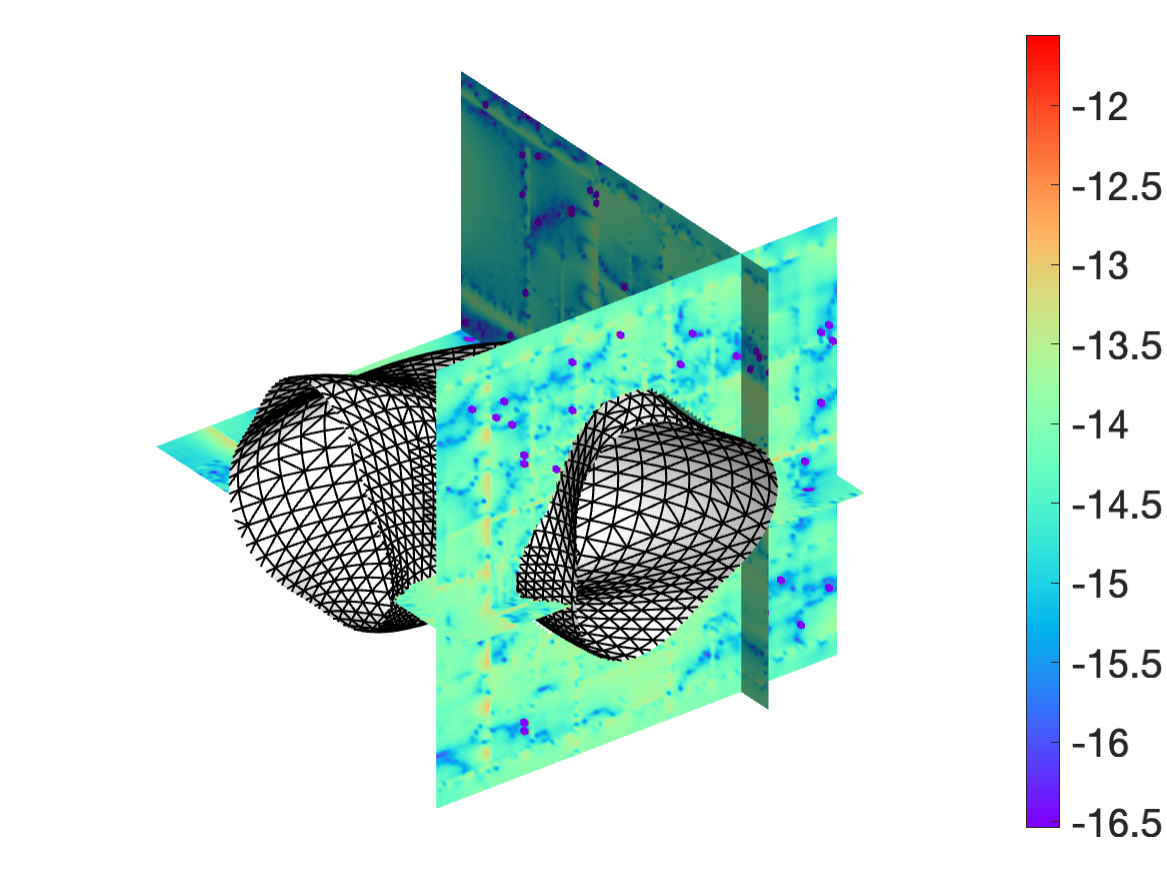}
\caption{\sf Slice plots of $\log_{10}$ of the pointwise relative errors
  for the warped torus.
  Left: $p=8,\ n_{\theta} = 36,\ n_{\phi} = 72$.
  Right: $p=14,\ n_{\theta} = 48,\ n_{\phi} = 96$.}
\label{fig:torussliceplot}
\end{figure}

The torus is discretized into $n_{\theta}\times n_{\phi}$ equi-sized rectangles, with
each further subdivided into two equal triangular patches in the parameter space.
Thus, the total number of patches $N_{\rm patches}$ is $2n_{\theta} n_{\phi}$. 
The equispaced tensor grid points are placed on the rectangular box
$[-0.5,2]\times[-0.5,2]\times[-1.25,1.25]$ and numerical solution is evaluated at
$892932$ points in the exterior of the warped torus.
\Cref{table:torus} shows $E_\infty$ defined in \eqref{relativeerror}
at sample values of $n_{\theta}\times n_{\phi}$ for $p=4, 6, \ldots, 14$.
The image on the right side of \cref{fig:torus} plots these data points as
a function of $N_{\rm patches}$ for these six values of $p$, with dashed lines showing
the expected convergence order. \Cref{fig:torussliceplot} plots pointwise
relative errors at three orthogonal slices where such errors are maximal.

\begin{table}[!ht]
\caption{\sf Relative $l_\infty$ errors $E_\infty$ for the warped torus example at
  sample values of $n_{\theta}$, $n_{\phi}$, and $p$. The total number of patches is
  $N_{\rm patches}=2n_{\theta}n_{\phi}$ and the total number of discretization points
  is $N=N_{\rm patches} p(p+1)/2$.}
\centering
\begin{tabular}{|c|c|c|c|c|c|}
\hline
{\diagbox{$n_{\theta}\times n_{\phi}$}{$p$}} & 6 & 8 & 10 & 12 & 14\\
\hline
12$\times$24  & 2.83e-01 & 4.96e-03 & 2.36e-03 & 1.86e-03 & 9.31e-03\\
18$\times$36  & 8.37e-04 & 8.22e-05 & 1.32e-04 & 5.59e-05 & 9.62e-05\\
24$\times$48  & 3.89e-05 & 4.69e-06 & 2.72e-06 & 6.57e-06 & 9.55e-07\\
30$\times$60  & 3.82e-05 & 1.40e-07 & 2.37e-07 & 5.25e-07 & 3.04e-07\\
36$\times$72  & 3.29e-06 & 6.87e-09 & 2.20e-09 & 3.68e-08 & 8.11e-09\\
42$\times$84  & 6.52e-06 & 5.45e-09 & 3.56e-10 & 9.47e-10 & 6.83e-11\\
48$\times$96  & 1.25e-06 & 9.17e-10 & 9.18e-11 & 5.02e-11 & 5.72e-12\\
54$\times$108 & 2.67e-07 & 5.11e-10 & 7.66e-11 & 6.11e-11 & 8.09e-12\\
60$\times$120 & 7.60e-08 & 2.04e-10 & 9.00e-12 & 2.68e-11 & 5.90e-12\\
\hline
\end{tabular}\label{table:torus}
\end{table}

\subsubsection{Example 3: A cushion-shaped geometry}
Next, the boundary is a cushion-shaped surface 
\be\label{cushion}
\br(\theta,\phi) = \begin{bmatrix}
  f(\theta,\phi)\cos(\theta)\cos(\phi)\\
  f(\theta,\phi)\sin(\theta)\cos(\phi)\\
  f(\theta,\phi)\sin(\phi)
 \end{bmatrix}, \quad (\theta,\phi) \in [0,\pi]\times[0,2\pi],
\ee
where the radial function is $f(\theta,\phi) = \left(4/5+1/2\left(\cos(2\theta)-1\right)\left(\cos(4\phi)-1\right)\right)^{1/2}$. See the image on the left side of \cref{fig:cushion}. We use the ``cubed sphere'' in \cite{ronchi1996jcp} to divide the cushion into
six big charts, with each chart further divided into $2n_{\theta}\times n_{\phi}$ triangles. Thus,
the total number of patches $N_{\rm patches}$ is $12n_{\theta}n_{\phi}$. \Cref{table:cushion_error}
shows relative $l_\infty$ error at sample data points of $N_{\rm patches}$ for $p=4,6,\ldots,14$. The image on the right side of \cref{fig:cushion} plots the numerical convergence order with these data points and expected convergence order in dashed lines. Finally,
\cref{fig:cushionsliceplot} show the pointwise relative errors at three orthogonal
slices.

\begin{figure}[!ht]
\centering
\includegraphics[height=45mm]{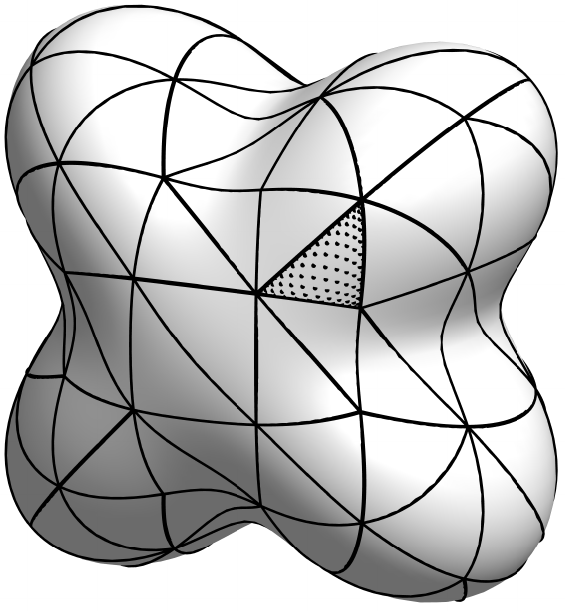} 
\hspace*{2mm}
\includegraphics[height=45mm]{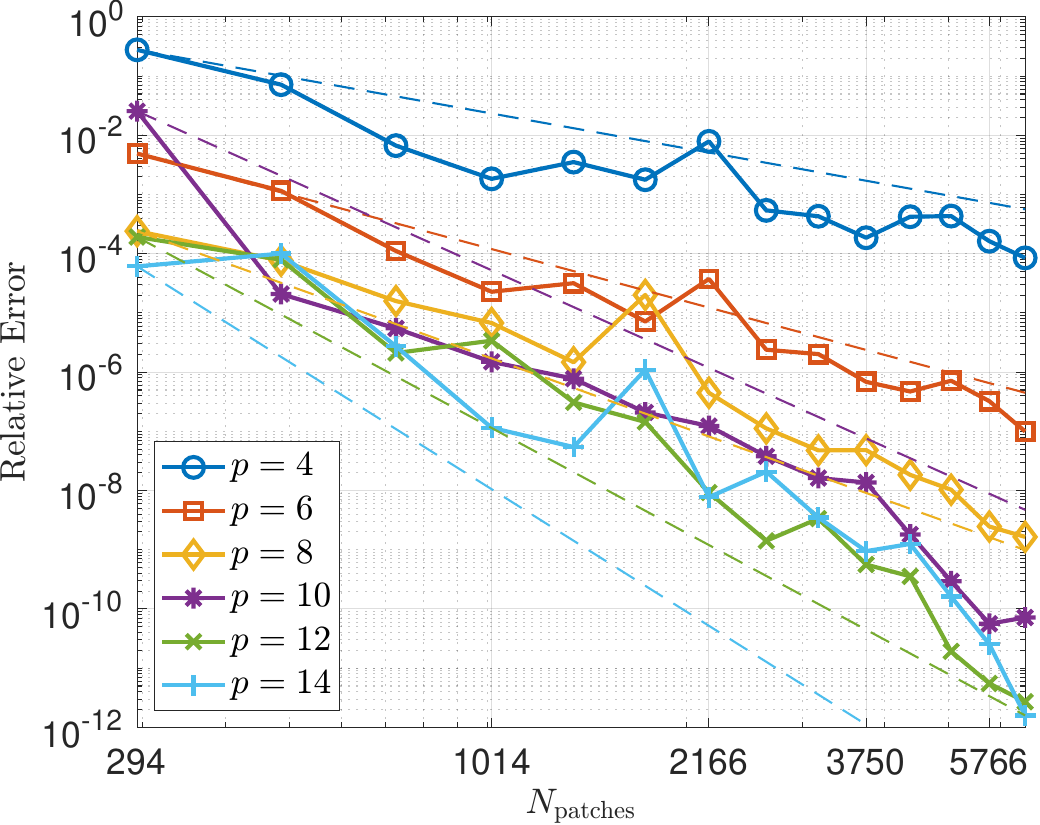}
\caption{\sf Left: triangulated cushion-shaped boundary. Right:
  relative $l_\infty$ errors as functions of $N_{\rm patches}$ for various orders $p$.}
\label{fig:cushion}
\end{figure}

\begin{figure}[!ht]
\centering
\includegraphics[height=45mm]{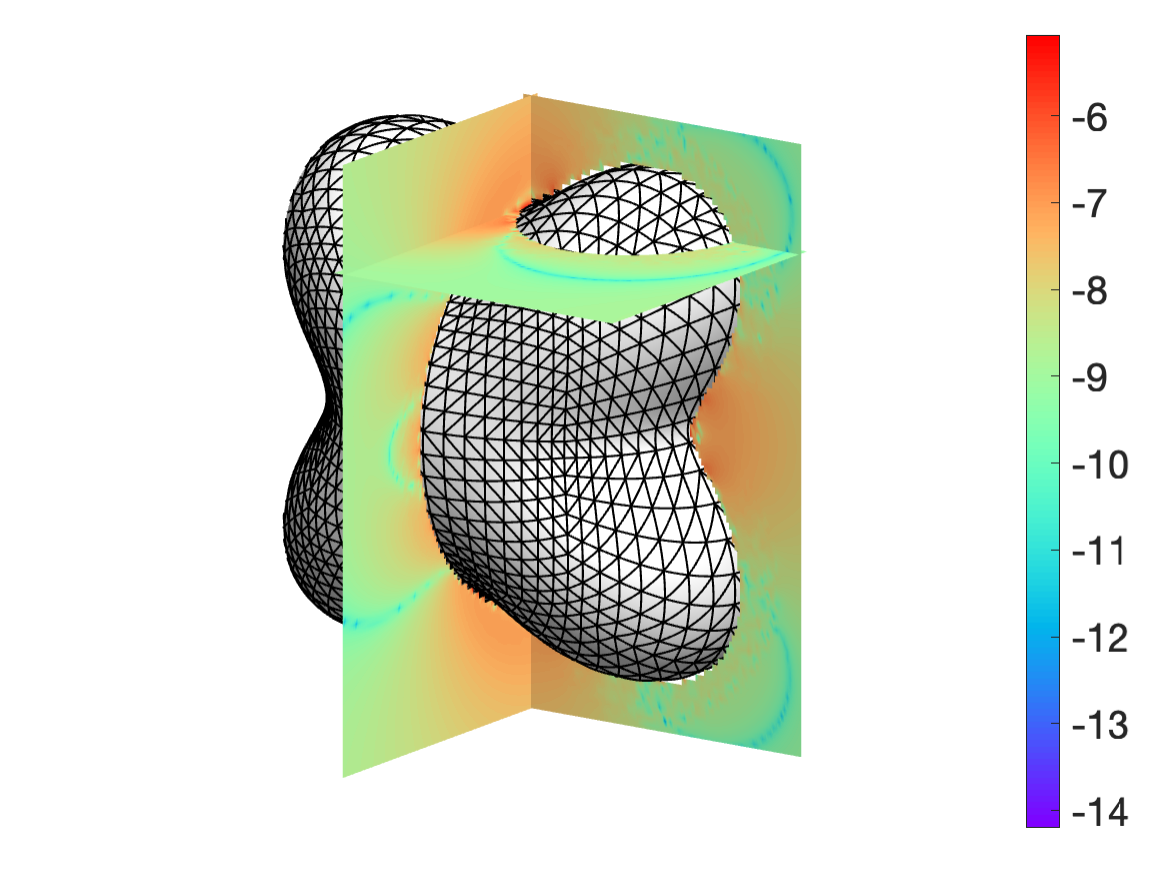}
\hspace*{2mm}
\includegraphics[height=45mm]{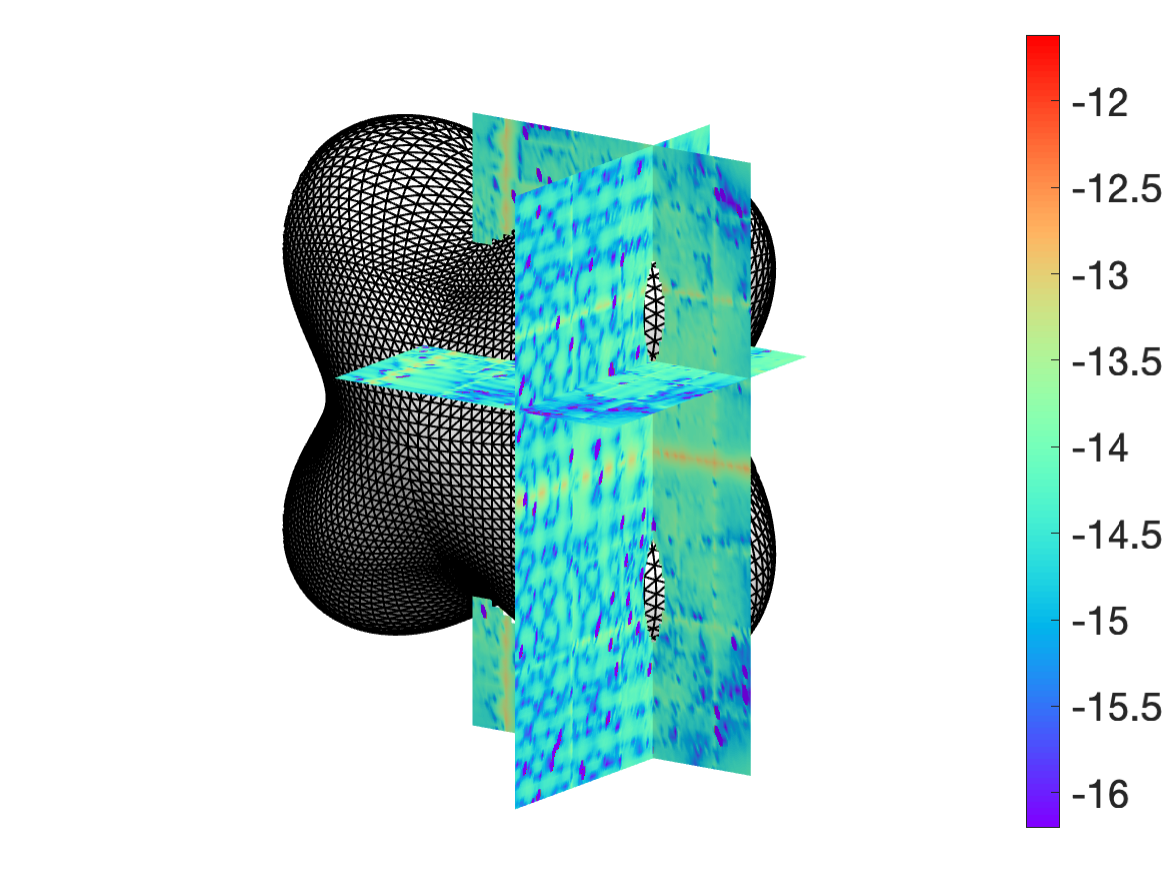}
\caption{\sf Slice plots of $\log_{10}$ of the pointwise relative errors for
  the cushion-shape boundary.
  Left: $p=6, n_{\theta} = 17, n_{\phi} = 17$. 
  Right: $p=12, n_{\theta} = 33, n_{\phi} = 33$.}
\label{fig:cushionsliceplot}
\end{figure}

\begin{table}[t]
\caption{\sf Relative $l_\infty$ errors $E_\infty$ for the cushion shaped example
  at sample values of $n_{\theta}$, $n_{\phi}$, and $p$. The total number of patches is
  $N_{\rm patches}=12n_{\theta}n_{\phi}$ and the total number of discretization points
  is $N=N_{\rm patches} p(p+1)/2$.}
\centering
\begin{tabular}{|c|c|c|c|c|c|}
\hline
{\diagbox{$12\times n_{\theta}n_{\phi}$}{$p$}} & 6 & 8 & 10 & 12 & 14\\
\hline
$12\times 9^2$    & 9.75e-04 & 6.53e-05 & 2.16e-05 & 7.57e-05 & 8.25e-05 \\ 
$12\times 13^2$   & 2.31e-05 & 6.11e-06 & 1.74e-06 & 3.60e-06 & 1.03e-07 \\ 
$12\times 17^2$   & 8.76e-06 & 1.64e-05 & 1.74e-07 & 1.48e-07 & 1.22e-06 \\ 
$12\times 21^2$   & 2.62e-06 & 1.32e-07 & 3.06e-08 & 1.26e-09 & 2.07e-08 \\ 
$12\times 25^2$   & 6.17e-07 & 3.94e-08 & 1.19e-08 & 4.41e-10 & 8.74e-10 \\ 
$12\times 29^2$   & 6.84e-07 & 9.26e-09 & 2.64e-10 & 1.54e-11 & 1.41e-10 \\ 
$12\times 33^2$   & 1.21e-07 & 1.42e-09 & 6.22e-11 & 2.44e-12 & 1.44e-12 \\ 
\hline
\end{tabular}\label{table:cushion_error}
\end{table}	

In Examples 2 and 3, we have used a more or less uniform triangulation
in the parameter space. However, both geometries have some fine details that
could be best resolved with some adaptivity employed in the triangulation.
The somewhat noisy data points in the convergence plots are likely due to
the under-resolved nonadaptive triangulation.
We would also like to remark that 
most figures and tables in \cite{zhu2021thesis,zhu2022sisc}
show the convergence for $p$ up to $7$ except Fig.~3, where values of $p$ up to $10$
are shown. While the original scheme in \cite{zhu2021thesis,zhu2022sisc} can achieve about ten digits
of accuracy with adaptive integration for computing line integrals, our improved scheme
achieves about twelve digits of accuracy.

\subsection{Study of efficiency}
Here, we demonstrate the efficiency of our quadrature scheme during the solve
phase and the evaluation phase. For both cases, we use $T_{\rm FMM}$ as the benchmark
to measure the relative throughput of our quadrature scheme in points
processed per second per core.

\subsubsection{Example 4: Green's identity on a stellarator geometry}
\begin{figure}[t]
\centering
\includegraphics[height=45mm]{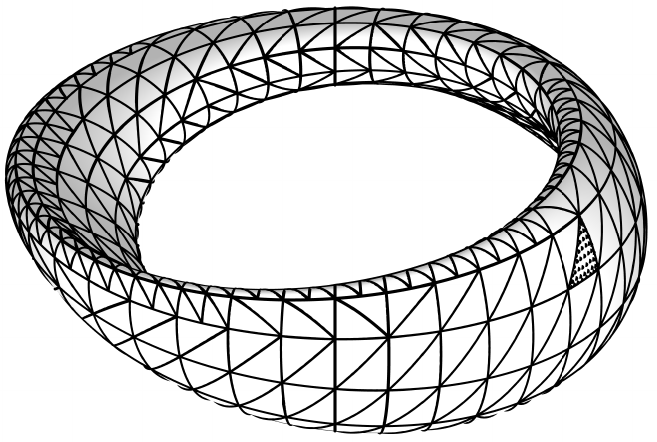}
\caption{\sf Triangulated stellarator-shaped geometry.}
\label{fig:stellarator}
\end{figure}

In the fourth example, we demonstrate the performance of our scheme by verifying the Green's identity
\be\label{greenformula}
\frac{1}{2}u(\brp) = \slp\left[\frac{\partial u}{\partial n}\right](\brp)-\dlp[u](\brp),
\quad \brp\in S.
\ee
The boundary is the surface of a typical stellarator
\begin{equation}\label{stellarator}
\br(u,v) = \sum_{i=-1}^{2} \sum_{j=-1}^{1} \delta_{i,j}
\begin{bmatrix}
  \cos{v} \cos{((1-i)\, u+j\, v)} \\
  \sin{v} \cos{((1-i) \, u+j \, v)} \\
  \sin{((1-i)\, u+j \, v)} 
\end{bmatrix} \, , \quad (u,v)\in [0,2\pi]^2,
\end{equation}
where the non-zero coefficients are $\delta_{-1,-1}=0.17$, 
$\delta_{-1,0} = 0.11$, $\delta_{0,0}=1$, $\delta_{1,0}=4.5$, $\delta_{2,0}=-0.25$, $\delta_{0,1} = 0.07$, and $\delta_{2,1}= -0.45$. See the image on the left side of \cref{fig:stellarator}. The triangulation is carried out by dividing the square in the parameter space into
$n_u\times n_v$ equi-sized rectangles, with each rectangle further subdivided into two
equal triangles. The stellarator has high curvature regions that rotate along one of
the angular directions, introducing significant challenges for the accurate evaluation
of layer potentials.

\begin{figure}
\centering
\includegraphics[height=45mm]{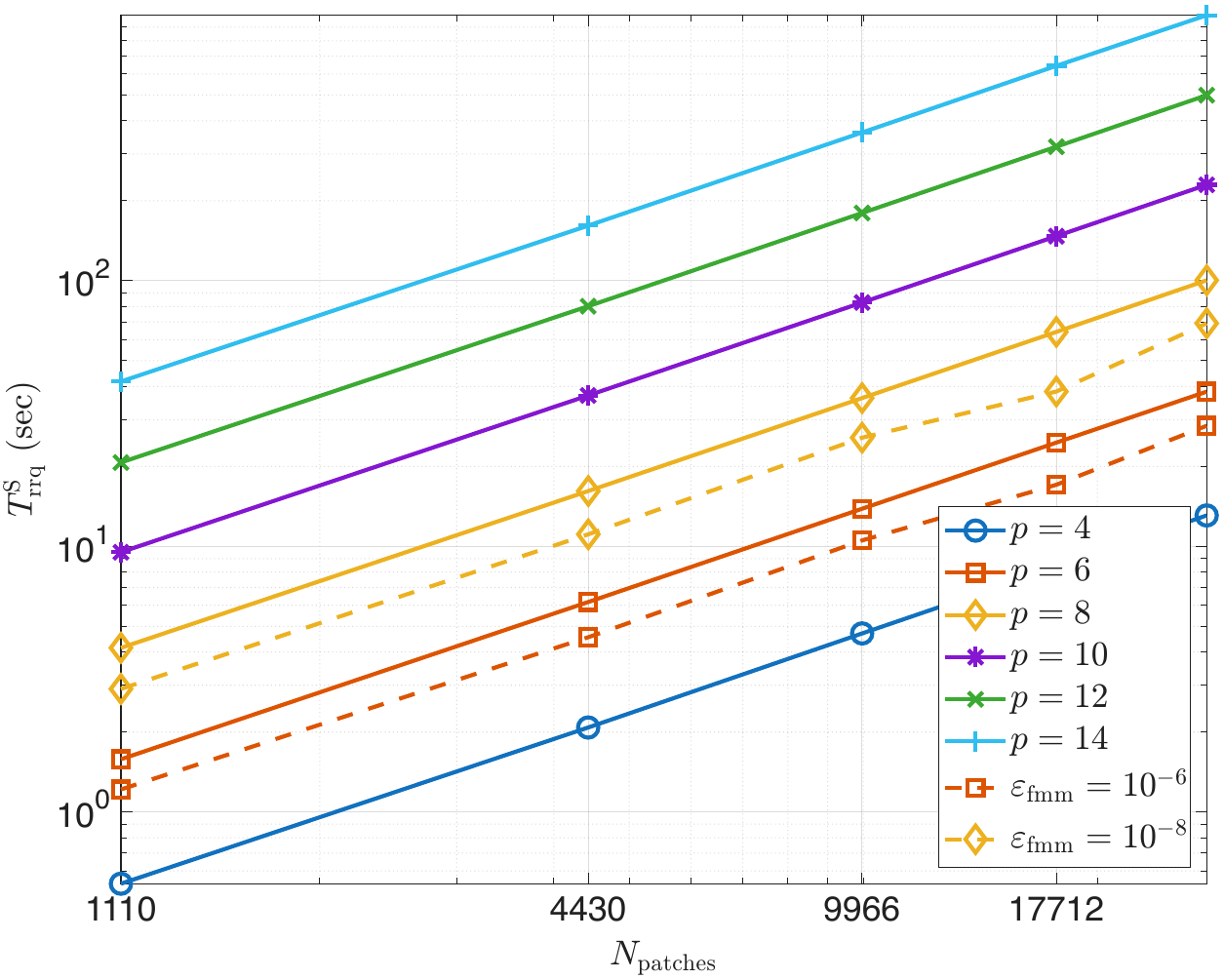}
\hspace*{2mm}
\includegraphics[height=45mm]{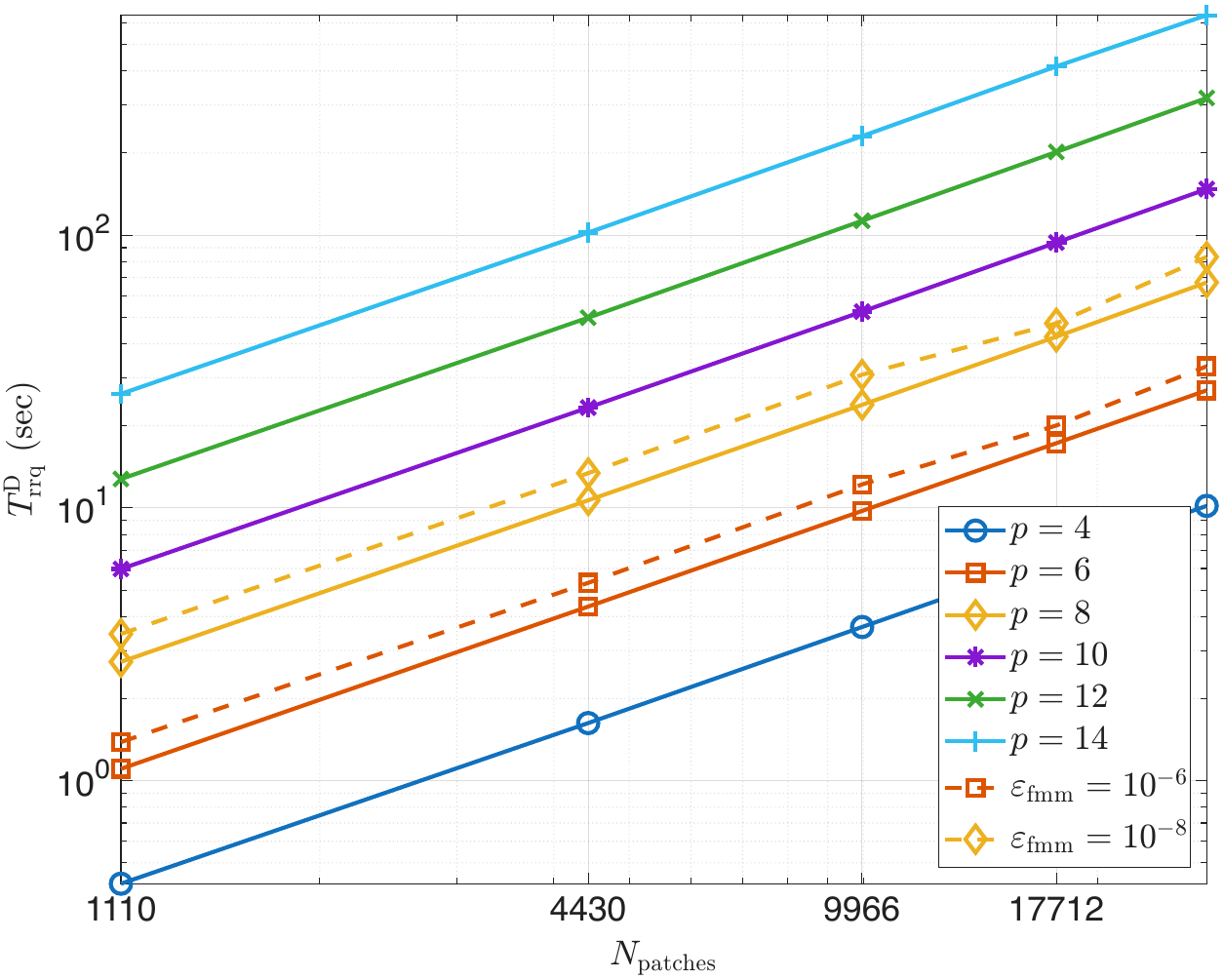}
\caption{\sf Computation time in seconds in the quadrature correction part
  of Laplace layer potentials for the stellarator shaped boundary in
  \cref{fig:stellarator} for $p=4,6,\ldots,14$.
  Left: Computation time $T_{\rm RRQ}$ for the Laplace single layer potential operator.
  Right: Computation time $T_{\rm RRQ}$ for the Laplace double layer potential operator.
  \Bl In both figures, the dashed lines show the computation time of a single FMM call:
  $\varepsilon_{\rm FMM}=10^{-6}$ on the $p=6$ discretization and
  $\varepsilon_{\rm FMM}=10^{-8}$ on the $p=8$ discretization.
  The total number of discretization points is $N = N_{\rm patches} \cdot p(p+1)/2$.\Bk
}
\label{fig:stellaratortiming}
\end{figure}

\Cref{fig:stellaratortiming} plots the computation time in the quadrature correction part
of our scheme as a function of $N_{\rm patches}$ for the stellarator shaped boundary.
\Cref{table:stellaratorthroughput} reports the throughput (in points processed
per second per core)
of our quadrature scheme. The first column lists the approximation order $p$. The second column lists the throughput of our quadrature scheme
for the Laplace single layer potential operator, i.e.,
\be
X^{S}_{\rm RRQ} = N/T^{S}_{\rm RRQ},
\ee
where $N$ is the total number of discretization points on the boundary and
$T^{S}_{\rm RRQ}$ is the time for building the sparse quadrature correction matrix
for the SLP. The third column lists the throughput of a single FMM \Bl call \Bk
for evaluating the smooth part
of the SLP, i.e.,
\be
X^{S}_{\rm FMM} = N/T^{S}_{\rm FMM},
\ee
where $T^{S}_{\rm FMM}$ is the time on applying $\bS_{\rm smooth}$ to a given vector using
the FMM. The fourth column lists the throughput of our quadrature scheme
for the Laplace double layer potential operator. The fifth column lists the throughput
of a single FMM call for evaluating the smooth part
of the DLP. The sixth column lists the tolerance used in the FMM.
And the last column lists the relative $l_\infty$ error for varifying Green's identity
in \cref{greenformula}.

We would like to compare this table with the previous
state-of-the-art optimized adaptive quadrature scheme in \cite{greengard2021jcp}.
In \cite[Table 1(c), (d)]{greengard2021jcp}, the speeds for $p$ up to $8$ are reported.
At about six digits of accuracy, the throughput for $p=8$ for the Helmholtz single layer
potential in \cite{greengard2021jcp} is $945$ points per second per core, the throughput
for the Helmholtz FMM is about $5960$ points per second per core. That is, the throughput
in the quadrature correction is approximately six times slower than the FMM.
On the other hand, the throughput in the quadrature correction of our scheme
is $9921$ per second per core for the Laplace single layer potential operator and
$15017$ per second per core
for the Laplace double layer potential operator, both of which are close to the FMM speed.
Note that if iterative solvers such as GMRES are used to solve the resulting linear
system, then the sparse matrices $\bS_{\rm RRQ}$ and $\bD_{\rm RRQ}$ only need to be built
once, while the FMM has to be applied $N_{\rm iter}$ times. Here, $N_{\rm iter}$ is the
number of GMRES iterations to reach the prescribed tolerance, which ranges from $10$ to
low hundreds for well-conditioned problems, but could be much larger for ill-conditioned
problems.

\begin{table}[t]
\caption{\sf Throughput in points processed per second per core of our quadrature scheme
during the solve phase for the stellarator shaped boundary.}
\centering
\begin{tabular}{|c|c|c|c|c|c|c|}
\hline
$p$ & $X^{S}_{\rm RRQ}$ & $X^{S}_{\rm FMM}$ & $X^{D}_{\rm RRQ}$ & $X^{D}_{\rm FMM}$ &  $\varepsilon_{\rm FMM}$ & $E_\infty$ \\
\hline
4  & 21285 & 26894 & 27212 & 24144 & 1.00e-04 &  1.08e-04 \\ 
6  & 15065 & 20445 & 21376 & 17494 & 1.00e-06 &  2.07e-05 \\ 
8  & 9921  & 13983 & 15017 & 11646 & 1.00e-08 &  2.76e-07 \\ 
10 & 6621  & 13341 & 10449 & 10803 & 1.00e-09 &  3.15e-08 \\ 
12 & 4324  & 9888  & 6845  & 8089  & 1.00e-10 &  7.59e-10 \\ 
14 & 2907  & 9113  & 4534  & 7178  & 1.00e-12 &  9.74e-12 \\ 
\hline
\end{tabular}
\label{table:stellaratorthroughput}
\end{table}

\begin{remark}
  Although we use the FMM as a benchmark to demonstrate the performance of our
  quadrature scheme, its efficiency is also quite important for fast direct
  solvers~\cite{fds1,rskelf1,fds16,fds2,fds3,kenho1,ho2016cpam1,ho2016cpam2,fds5,fds6,fds7,fds8,minden2017mms,rskelf2}.
  In these solvers, generating matrix entries for discretized integral operators
  consumes a significant amount of time when constructing sparse representations
  of the system matrix and its inverse.
\end{remark}

\subsubsection{Example 5: Interlocking tori}
\begin{figure}[!ht]
\centering
\includegraphics[height=45mm]{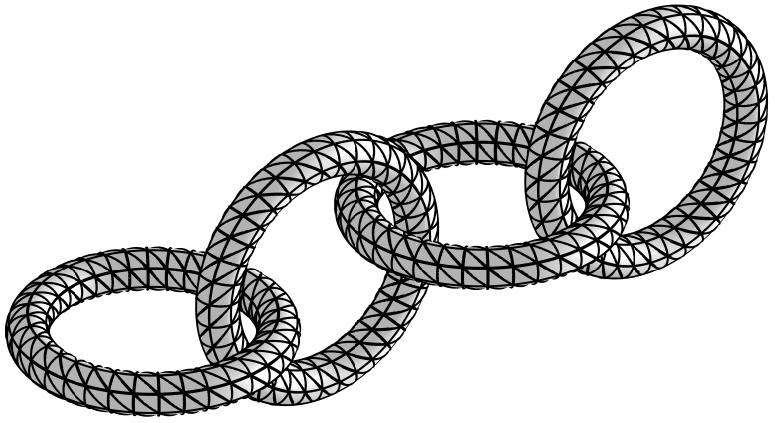}
\caption{\sf Boundary with four interlocking tori.}
\label{fig:iltori}
\end{figure}

Next, we assess the performance of our quadrature scheme during the evaluation phase,
where layer potentials are evaluated in a volumetric grid. That is, $N_S$
sources are on the
boundary and $N_T$ targets are in the entire computational domain with $N_T\gg N_S$.
Such tasks are often encountered 
when the solution to the PDE in the whole volume is required, or when the PDE has
an inhomogeneous term, i.e., the governing equation is the Poisson equation instead
of the Laplace equation (see, for example, \cite{epstein2025sinum}).
We set up a complex geometry test example where the predominant cost of layer
potential evaluation comes from computing nearly singular integrals.
The boundary consists of four interlocking tori, where
each torus is a shifted and rotated version of
\eqref{eq:torus}, with parameters $a=3$, $b=0.5$, $\omega_c=0$, $\omega_n=0$ and
$\omega_m=0$. A translation of $4.95$ along the $x$-direction
is applied so that the shortest distance between two adjacent tori is $0.05$.
See \cref{fig:iltori}.

Each torus surface is discretized into four different configurations:
$6\times 36$, $12\times 72$, $18\times 108$, and $24\times 144$ rectangules, with each
rectangle further subdivided into two equal triangular patches in the parameter space.
For each discretization, we solve the exterior Dirichlet problem \cref{exteriordirichlet},
then evaluate the solution $u$ 
on a corresponding uniform grid of  $101\times 101\times 101$, $201\times 201\times 201$, $301\times 301\times 301$,
and $401\times 401\times 401$ points within the bounding box $[-5,20] \times [-6,6] \times [-6, 6]$ enclosing the boundary.
The total number of target points in the exterior of the interlocking tori is $N_T = [1013571, 7988041, 26825787, 63428399]$
for these four cases, out of which there are $N_T^{c}=[328492, 1198758, 2627046, 4613672]$ close target points that
require quadrature correction, respectively.

\begin{figure}[!ht]
\centering
\includegraphics[height=45mm]{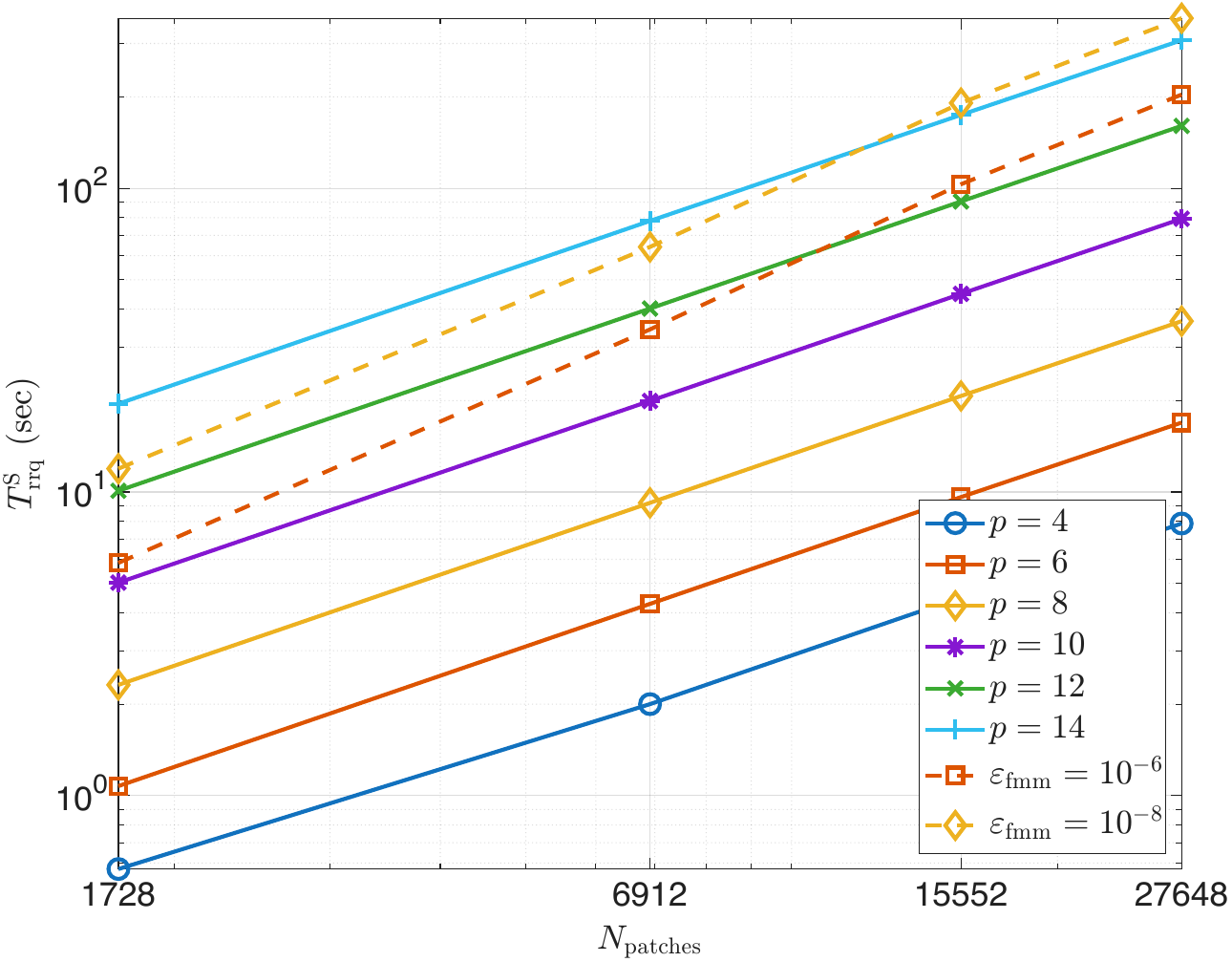}
\hspace*{2mm}
\includegraphics[height=45mm]{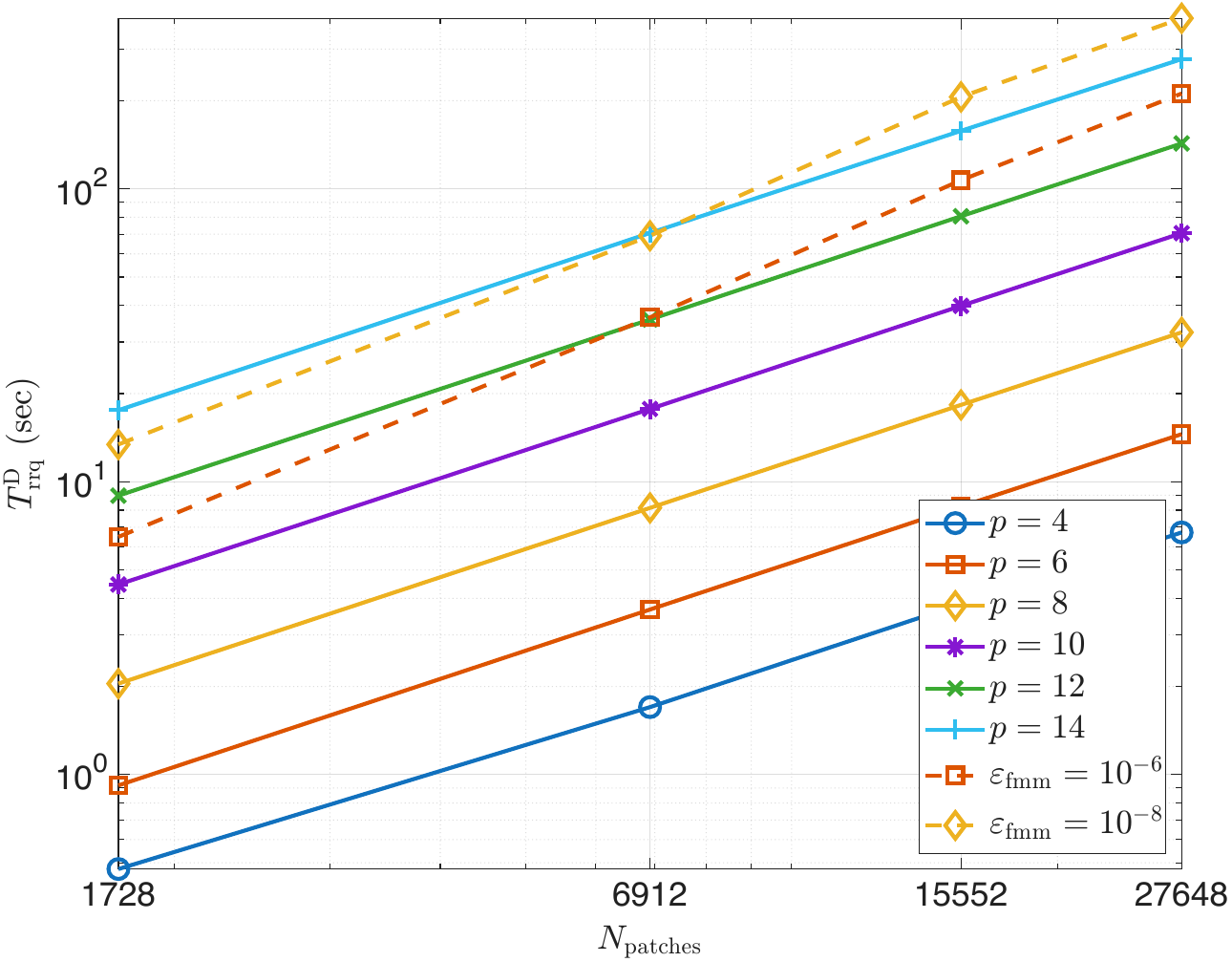}
\caption{\sf \Bl Same as \cref{fig:stellaratortiming} but for the
volumetric evaluation of Laplace layer potentials for the interlocking-tori boundary in
\cref{fig:iltori}.\Bk
}
\label{fig:iltoritiming}
\end{figure}

\Cref{fig:iltoritiming} plots the computation time in the quadrature correction part
of our scheme as a function of $N_{\rm patches}$ for the interlocking-tori boundary.
\Cref{table:iltorierror} reports the throughput (in points processed per second per core)
of our quadrature scheme when layer potentials are viewed as maps from the boundary
to the volume. Here the throughput for the quadrature correction part
is defined via the formulas
\be
X^{S}_{\rm RRQ} = N_T^{c}/T^{S}_{\rm RRQ},
\quad
X^{D}_{\rm RRQ} = N_T^{c}/T^{D}_{\rm RRQ},
\ee
and the throughput for the FMM part is defined via the formulas
\be
X^{S}_{\rm FMM} = N_T/T^{S}_{\rm FMM},
\quad
X^{D}_{\rm FMM} = N_T/T^{D}_{\rm FMM}.
\ee
Note that $N_T^{c}$ is a small fraction of $N_T$. Thus, the time spent on quadrature
correction is a small fraction of that on the FMM, even though the throughput are
close to each other.

\begin{table}[!ht]
\caption{\sf Throughput in points processed per second per core of our quadrature scheme
during the evaluation phase for the interlock-tori boundary.}
\centering
\begin{tabular}{|c|c|c|c|c|c|c|}
\hline
$p$  & $X^{S}_{\rm RRQ}$ & $X^{S}_{\rm FMM}$ & $X^{D}_{\rm RRQ}$ & $X^{D}_{\rm FMM}$ &
$\varepsilon_{\rm FMM}$ & $E_\infty$ \\
\hline
4  & 583929 & 468224 & 687077 & 433446 & 1.00e-04 &  1.14e-04 \\ 
6  & 272680 & 260129 & 318736 & 250650 & 1.00e-06 &  1.41e-06 \\ 
8  & 126596 & 140465 & 143307 & 130197 & 1.00e-08 &  3.22e-08 \\ 
10 & 58368  & 108913 & 65770  & 102488 & 1.00e-09 &  7.47e-10 \\ 
12 & 28941  & 74238  & 32493  & 69085  & 1.00e-10 &  1.33e-11 \\ 
14 & 14995  & 55063  & 16584  & 50249  & 1.00e-12 &  6.11e-12 \\ 
\hline
\end{tabular}
\label{table:iltorierror}    
\end{table}

\section{Conclusions and further discussions}\label{sec:conclusions}
We have developed a recursive reduction quadrature (RRQ) scheme for evaluating
Laplace layer potentials on surfaces in three dimensions. The RRQ scheme starts
from a quaternion harmonic polynomial approximation of the density for the Laplace
double layer potential. Since harmonic polynomials are defined over the entire
ambient space, this approximation naturally extends to a small volumetric
neighborhood of the surface patch. Moreover, the quaternion harmonic polynomial
approximation is constructed so that the general Stokes theorem can be applied
to reduce the surface integrals to line integrals over the boundary of the patch.
Certain degrees of freedom exist in choosing these line integrals, and we have
adopted target-centered formulas to derive these 1-forms. The resulting line
integrals have exactly the same singularities as the kernels in the original
surface integrals. We then apply singularity-swap quadrature to evaluate these
line integrals recursively. In other words, line integrals are evaluated via
integration by parts, this time by extending the real parameter of the boundary
curve into the complex plane. Altogether, we have been able to reduce the evaluation
of singular and nearly singular surface integrals to function evaluations at the
vertices of the triangular patch. The resulting scheme can handle self and near
interactions, as well as close evaluations, with unparalleled accuracy and efficiency.

As observed in \cite{zhu2021thesis,zhu2022sisc}, the closest relative of
our quadrature scheme is the {\it kernel-split quadrature} developed by
Johan Helsing and his
collaborators~\cite{helsing2011sisc,helsing2015acom,helsing2018sisc,helsing2014jcp,helsing2018jcp,helsing2008jcp2,helsing2008jcp1}, which has been extended to evaluating
almost all layer potentials in two dimensions. Indeed,
extending RRQ to other layer potentials in three dimensions,
such as the Helmholtz, Stokes, and elastic layer potentials, is straightforward.
Moreover, the recursive reduction approach used in our quadrature construction
is independent of the shapes of patches and can thus be easily extended to,
for example, quadrilateral patches. We are currently working on these extensions,
and the findings will be reported at a later date.

\section{Acknowledgments}
The authors would like to thank Charles Epstein and Leslie Greengard
for helpful discussions. \Bl We also thank the anonymous referees for their helpful suggestions, which have improved the quality of the presentation.\Bk

\bibliographystyle{siam}
\bibliography{journalnames,fmm,quadrature}

\end{document}